\documentclass[11pt]{amsart}
\usepackage{amsmath, amsthm, amscd, amsfonts,amssymb, graphicx}
\usepackage{accents}
\usepackage{xcolor}
\usepackage{hyperref}
\newlength{\dhatheight}

\newcommand{\bea}{\begin{eqnarray*}}
\newcommand{\eea}{\end{eqnarray*}}

\newcommand{\beq}{\begin{equation}}
\newcommand{\eeq}{\end{equation}}
\newcommand{\bfell}{\mbox{\boldmath $\ell$ \unboldmath} \hskip -0.05 true in}

\newcommand{\dd}{{\rm d}}

\newcommand{\ii}{\mathrm{i}}

 \newcommand{\D}{\mathrm{d}}

\newtheorem{theorem}{Theorem}[section]
\newtheorem{lemma}[theorem]{Lemma}
\theoremstyle{definition}

\newtheorem{proposition}[theorem]{Proposition}
\newtheorem{corollary}[theorem]{Corollary}

\theoremstyle{remark}
\newtheorem{remark}[theorem]{Remark}

\numberwithin{equation}{section}

\begin{document}

\title[Non-Abelian Fourier Analysis on $\mathbf{\Gamma}\backslash SE(d)$]{Non-Abelian Fourier Analysis on $\mathbf{\Gamma}\backslash SE(d)$}

\author[A. Ghaani Farashahi]{Arash Ghaani Farashahi$^{1,*}$}
\address{$^1$Department of Mechanical Engineering, University of Delaware, USA.}
\email{aghf@udel.edu}
\email{ghaanifarashahi@outlook.com}

\author[G.S. Chirikjian]{Gregory S. Chirikjian$^2$}
\address{$^2$Department of Mechanical Engineering, University of Delaware, USA.}
\email{gchirik@udel.edu}
\address{$^2$Department of Mechanical Engineering, National University of Singapore, Singapore.}
\email{mpegre@nus.edu.sg}


\date{\today}

\thanks{$^*$Corresponding author}
\thanks{E-mail addresses: aghf@udel.edu (Arash Ghaani Farashahi) and gchirik@udel.edu (Gregory S. Chirikjian)}

\begin{abstract}
This paper presents a systematic study for the general theory of non-Abelian Fourier series of integrable functions on the homogeneous space $\mathbf{\Gamma}\backslash SE(d)$,  where $SE(d)$ is the special Euclidean group in dimension $d$, and $\mathbf{\mathbf{\Gamma}}$ is a discrete and co-compact subgroup of $SE(d)$.  Suppose that $\mu$ is the finite $SE(d)$-invariant measure on the right coset space $\mathbf{\Gamma}\backslash SE(d)$, normalized with respect to Weil's formula.  The analytic aspects of the proposed method works for any given orthonormal basis of the Hilbert function space $L^2(\mathbf{\Gamma}\backslash SE(d),\mu)$.  The paper is concluded with some convolution and Plancherel formulas.

%\date{Last Update \today}
\end{abstract}

\maketitle

%\tableofcontents

\section{{\bf Introduction}}

Special Euclidean groups are classical algebraic tools for modeling of the group of rigid body motions on Euclidean spaces, and play significant roles in finite dimensional geometric analysis, quantum mechanics, and coherent state analysis \cite{Bern.Ty.SIAM,  Fuhr,  Kisil.book,  Pere.Coh}. Over the last decades, different applied and computational aspects of constructive approximation methods including Fourier type analysis on special Euclidean groups have obtained significant attention in application areas, see \cite{Ba.Gi, PI5, Du,  Les, Y1, Y2}.  Specifically,  the right coset spaces of orthogonal lattices in special Euclidean groups of different dimensions,  such as $\mathbb{Z}^2\backslash SE(2)$ and $\mathbb{Z}^3\backslash SE(3)$,  appears as the configuration space in different applications in computational science and engineering including computer vision, robotics, mathematical crystallography, computational biology, and material science \cite{PI3, PI.Kya.2000, PI7, Kya.PI.2000, shifman, etal.PI.2006}.

Let $d\ge2$ be an integer and $SE(d)$ be the group consists of direct Euclidean isometries of $\mathbb{R}^d$.   
Suppose that $\mathbf{\Gamma}$ is a discrete and co-compact subgroup of $SE(d)$ and $\mathbf{\Gamma}\backslash SE(d)$ is the right coset space of $\mathbf{\Gamma}$ in $SE(d)$. 
Invoking the algebraic structure of the group $SE(d)$ and constructive classifications of discrete co-compact subgroups of $SE(d)$, the subgroup $\mathbf{\Gamma}$ is not normal in $SE(d)$ and hence the right coset space $\mathbf{\Gamma}\backslash SE(d)$ is not a group in general.  In the case of $d=2$, fundamental properties of a unified approach for noncommutative Fourier series of continuous functions on $\mathbf{\Gamma}\backslash SE(2)$ generated by compactly supported functions on $SE(2)$ studied in \cite{AGHF.GSC.PAMQ}.  The discussed approach assumed some type of  Fourier  transform integrability condition on the dual space $(0,\infty)$ for generator functions on $SE(2)$. This paper generalizes these results in different directions.  To begin with,  we present the general theory for Euclidean group $SE(d)$ of arbitrary dimension.  In addition,  the presented generalization is valid for a larger class of integrable function and hence neither compactly supported assumption for generators nor integrability assumption for non-Abelian Fourier transform of generators on the dual space $(0,\infty)$ is required.  

This article contains 6 sections.  Section
2 is devoted to fixing notation and gives a
brief summary of non-Abelian Fourier analysis on 
the unimodular group $SE(d)$ and classical properties of functional analysis on 
the right coset space $\mathbf{\Gamma}\backslash SE(d)$.  In Section
3, we present a constructive approach for the general theory of non-Abelian Fourier analysis of functions on the right coset space $\mathbf{\Gamma}\backslash SE(d)$.
The presented approach implies a systematic study for harmonic analysis of non-Abelian Fourier series in the Hilbert function space $L^2(\mathbf{\Gamma}\backslash SE(d),\mu)$, where $\mu$ is the finite $SE(d)$-invariant measure on $\mathbf{\Gamma}\backslash SE(d)$,  normalized with respect to Weil's formula.  The next section discusses absolutely convergence of non-Abelian Fourier series 
and  associated reconstruction formulas of functions on the right coset space $\mathbf{\Gamma}\backslash SE(d)$,  including
the Fourier integral operator on $SE(d)$. 
Section 5 addresses convolution of functions on $SE(d)$ by functions on the right coset space $\mathbf{\Gamma}\backslash SE(d)$.  It is shown that  such a convolution defines a module action of the Banach algebra $L^1(SE(d))$ on the Banach function space $L^p(\mathbf{\Gamma}\backslash SE(d),\mu)$, for every $p\ge 1$.  Then analytic properties of non-Abelian Fourier series for convolution functions on the right coset space $\mathbf{\Gamma}\backslash SE(d)$ studied.  The paper is concluded by investigation of the presented results for the case of $d=2$ and $d=3$. 

%\newpage
\section{{\bf Preliminaries and Notation}}

Throughout this section we shall present preliminaries and the notation. 

\subsection{General notation} Let $\mathcal{H}$ be a Hilbert space. Then $\mathcal{U}(\mathcal{H})$ denotes the group of all unitary linear operators on $\mathcal{H}$.  If $T:\mathcal{H}\to\mathcal{H}$ is a bounded linear operator then $\|T\|_{\rm HS}$ (resp.  ${\rm tr}[T]$) denotes the Hilbert-Schmidt norm (resp.  trace) of $T$.  In addition,  $|T|$ is the positive square root of $T^*T$, where $T^*$ is the adjoint of $T$.  For more details on theory of operators and their properties we refer the reader to \cite{Mur}.

Suppose that $G$ is a locally compact group.  A unitary representation of $G$ is a homomorphism $\pi:G\to\mathcal{U}(\mathcal{H}_\pi)$ for some nonzero Hilbert space $\mathcal{H}_\pi$.  In this case,  $\mathcal{H}_\pi$ is called the representation space of $\pi$,  and dimension of $\mathcal{H}_\pi$ called the dimension or degree of $\pi$. The unitary representation $\pi:G\to\mathcal{U}(\mathcal{H}_\pi)$ is called continuous if the map $x\mapsto \pi(x)\zeta$ is continuous from $G$ to $\mathcal{H}_\pi$,  for every $\zeta\in\mathcal{H}_\pi$.  The unitary representation $\pi:G\to\mathcal{H}_\pi$ is called irreducible if it 
has no nontrivial invariant subspace.  

The set of equivalence classes of irreducible continuous unitary representations of a locally compact group $G$ is denoted by $\widehat{G}$ and called the (unitary) dual space of $G$.  We denote the equivalence class of an irreducible representation $\pi$ by $[\pi]$ or just $\pi$.  For more details on theory of unitary representations on groups and structure of their dual spaces we refer the reader to \cite{FollH} and references therein. 

\subsection{Harmonic analysis on $SE(d)$}
The special Euclidean group of $d$-dimensional space,  also called as Euclidean motion group,  $SE(d)$,  is the semi-direct product of $\mathbb{R}^d$
with the special orthogonal group $\mathbb{K}_d:=SO(d)$. That is,
\[
SE(d)=\mathbb{R}^d\rtimes SO(d)=SO(d)\ltimes\mathbb{R}^d.
\] 
We denote elements $g\in SE(d)$ as $g=(\mathbf{x},\mathbf{R})$ where $\mathbf{x}\in\mathbb{R}^d$ and $\mathbf{R}\in SO(d)$.  For every  $g=(\mathbf{x},\mathbf{R})$ and 
$g'=(\mathbf{x}',\mathbf{R}')\in SE(d)$ the group law is written as 
\[
g\circ g'=(\mathbf{x}+\mathbf{R}\mathbf{x}',\mathbf{RR}'),
\]
and 
\[
g^{-1}=(-\mathbf{R}^T\mathbf{x},\mathbf{R}^T),
\]
where $\mathbf{R}^T=\mathbf{R}^{-1}$.

Alternatively, the special Euclidean group $SE(d)$ can be faithfully  represented as the set of $(d+1)\times(d+1)$ homogeneous matrices of the form 
\begin{equation}
\mathbf{H}(g):=\left(\begin{array}{cc}
\mathbf{R} & \mathbf{x}  \\ 
\mathbf{0}^T & 1 \\ 
\end{array}\right)\hspace{1cm}{\rm for}\ \ g=(\mathbf{x},\mathbf{R}).
\end{equation}
Then the $SE(d)$ group law can be uniquely  identified via multiplication of matrices, that is 
\[
\mathbf{H}(g\circ g')=\mathbf{H}(g)\mathbf{H}(g'), \hspace{1cm}\ \mathbf{H}(g^{-1})=\mathbf{H}(g)^{-1}.
\]
The special Euclidean group $SE(d)$ is a unimodular group with the normalized Haar measure given by 
\[
\int_{SE(d)}f(g)\dd g:=\int_{SO(d)}\int_{\mathbb{R}^d}f(\mathbf{x},\mathbf{R})\dd\mathbf{x}\dd\mathbf{R},
\]
where $\dd\mathbf{x}$ is the normalized Lebesgue measure on $\mathbb{R}^d$ and $\dd\mathbf{R}$ is the probability Haar measure on $SO(d)$.

\subsection{Fourier analysis on $SE(d)$}
The special Euclidean group $SE(d)$ is a semi-direct product group with an Abelian normal factor.  Therefore,  classical methods for constructing unitary representations of semi-direct product groups could be applied,  known as the Mackey Machine,  e.g. Theorem 6.42 of \cite{FollH}.  

The set of all irreducible unitary representations, unitary dual,  of the Euclidean group $SE(d)$ can be classified as 
\[
\widehat{SE(d)}=\mathcal{R}_\#(SE(d))\cup\mathcal{R}_\infty(SE(d)),
\]
where $\mathcal{R}_\#(SE(d))$ (resp.  $\mathcal{R}_\infty(SE(d))$) denotes the set of all equivalence classes of finite (resp.  infinite) dimensional irreducible unitary representations of $SE(d)$, see \cite[Sec.  6.7, \S2]{FollH}.

The set $\mathcal{R}_\#(SE(d))$ can be characterized via the unitary dual $\widehat{\mathbb{K}_d}$ as follows
\[
\mathcal{R}_\#(SE(d))=\left\{\widetilde{\sigma}:\sigma\in\widehat{\mathbb{K}_d}\right\},
\]
where the lifted finite dimensional irreducible representation  $\widetilde{\sigma}:SE(d)\to\mathcal{U}(\mathcal{H}_\sigma)$ is given by 
\[
\widetilde{\sigma}(g):=\sigma(\mathbf{R}), \hspace{1cm}\forall \sigma\in\widehat{\mathbb{K}_d},
\]
for every $g=(\mathbf{x},\mathbf{R})\in SE(d)$.

The set $\mathcal{R}_\infty(SE(d))$ can be parametrized using $\mathcal{D}_d:=(0,\infty)\times\widehat{\mathbb{K}_{d-1}}$ via the following identification 
\[
\mathcal{R}_\infty(SE(d))=\left\{U_{p,\pi}:(p,\pi)\in\mathcal{D}_d\right\},
\]
where the infinite dimensional irreducible representation $U_{p,\pi}:SE(d)\to\mathcal{U}(\mathcal{K}_{\pi})$ is given by 
\[
U_{p,\pi}(g):=\mathrm{ind}_{H_d}^{SE(d)}(\xi_p\otimes\pi)(g),\hspace{1cm}\forall (p,\pi)\in\mathcal{D}_d,
\]
for every $g=(\mathbf{x},\mathbf{R})\in SE(d)$, where $H_d:=\mathbb{R}^d\rtimes \mathbb{K}_{d-1}$,  if $\mathbb{K}_{d-1}$ is considered as  the subgroup of $\mathbb{K}_d$ which leaves the vector $\mathrm{e}_1:=(\delta_{1,k})^T_{1\le k\le d}\in\mathbb{R}^d$ fixed,  and the character $\xi_p:\mathbb{R}^d\to\mathbb{T}$ is given by $\xi_p(\mathbf{x}):=e^{\ii p\langle\mathrm{e}_1,\mathbf{x}\rangle}=e^{\ii px_1}$,  for every $\mathbf{x}:=(x_1,\cdots,x_d)^T\in\mathbb{R}^d$.  
It should be mentioned that invoking construction of the induced representations,   if the representation $\pi\in\widehat{\mathbb{K}_{d-1}}$ is realised on the finite dimensional Hilbert space $\mathcal{H}_\pi$ then the representation $U_{p,\pi}$ can be canonically 
realised on an infinite dimensional Hilbert space $\mathcal{K}_{\pi}$ consists of $\mathcal{H}_\pi$-valued functions on $\mathbb{K}_d$,  for every $p\in(0,\infty)$.   For more details, on the structure of induced representation and properties of the representations $U_{p,\pi}$, we refer the readers to \cite{FollH, Ito, Klep.Lip, Vil}.

The group Fourier transform of each $f\in L^1(SE(d))$ at $(p,\pi)\in\mathcal{D}_d$ is given by 
\begin{equation}\label{hatp.pi}
\widehat{f}(p,\pi):=\int_{SE(d)}f(g)U_{p,\pi}(g^{-1})\dd g=\int_{SE(d)}f(g)U_{p,\pi}(g)^*\dd g.
\end{equation}
The corresponding convolution formula for functions $f_k\in L^1(SE(d))$ with $k\in\{1,2\}$ at $(p,\pi)\in\mathcal{D}_d$ is 
\begin{equation}\label{hatconvp.pi}
\widehat{(f_1\star f_2)}(p,\pi)=\widehat{f_2}(p,\pi)\widehat{f_1}(p,\pi).
\end{equation}

The Plancherel measure of $\widehat{SE(d)}$ is supported on the subset $\mathcal{R}_\infty(SE(d))$ and the Plancherel measure on $\mathcal{R}_\infty(SE(d))$ explicitly given by $C(d)p^{d-1}\dd p\otimes\kappa_\pi
$ where $C(d):=\frac{2}{2^{d/2}\Gamma(d/2)}$ ($\Gamma$ is the Gamma function) and $\kappa_\pi:=\dim\mathcal{H}_\pi$ for every $\pi\in\widehat{\mathbb{K}_{d-1}}$.  In addition,  the non-Abelian Fourier Plancherel/Parseval formula on the group $SE(d)$ is given by 
\begin{equation}\label{mainPL.SEd}
\int_{SE(d)}|f(g)|^2\dd g=C(d)\int_0^\infty\left(\sum_{\pi\in\widehat{\mathbb{K}_{d-1}}}\kappa_\pi\|\widehat{f}(p,\pi)\|_{\rm HS}^2\right)p^{d-1}\dd p,
\end{equation}
and the non-Abelian Fourier reconstruction formula on $SE(d)$ is 
\begin{equation}\label{main.SEd.rec}
f(g)=C(d)\int_0^\infty\sum_{\pi\in\widehat{\mathbb{K}_{d-1}}}\kappa_\pi\mathrm{tr}\left[\widehat{f}(p,\pi)U_{p,\pi}(g)\right]p^{d-1}\dd p,
\end{equation}
for $f\in L^1\cap L^2(SE(d))$ and $g\in SE(d)$, see \cite{Ku.Oka}.

Given $(p,\pi)\in\mathcal{D}_d$,  suppose that the irreducible unitary representation $U_{p,\pi}$ is realised on the infinite dimensional Hilbert space $\mathcal{K}_{\pi}$.  Let $U_p:SE(d)\to\mathcal{U}(\mathcal{K}_\infty)$ be the direct sum unitary representation of $SE(d)$ given by 
\[
U_p:=\bigoplus_{\pi\in\widehat{\mathbb{K}_{d-1}}}U_{p,\pi},
\]
which acts on the fixed Hilbert space $\mathcal{K}_\infty$,  that is the direct sum $\mathcal{K}_\infty:=\bigoplus_{\pi\in\widehat{\mathbb{K}_{d-1}}}\mathcal{K}_{\pi}$,  via 
\[
U_p(g)\zeta:=\{U_{p,\pi}(g)\zeta_\pi\}_{\pi\in\widehat{\mathbb{K}_{d-1}}},
\]
for every $\zeta=\{\zeta_\pi\}_{\pi\in\widehat{\mathbb{K}_{d-1}}}\in\mathcal{K}_\infty$.

The unified Fourier transform of each $f\in L^1(SE(d))$ at $p>0$ is given by 
\begin{equation}\label{hatp}
\widehat{f}(p):=\int_{SE(d)}f(g)U_p(g^{-1})\dd g=\int_{SE(d)}f(g)U_p(g)^*\dd g.
\end{equation}
In details,  if $f\in L^1(SE(d))$ and $p>0$ then $\widehat{f}(p)$ operates on $\mathcal{K}_\infty$ by 
\[
(\widehat{f}(p)\zeta)_\pi=\widehat{f}(p,\pi)\zeta_\pi,
\]
for every $\zeta=\{\zeta_\pi\}_{\pi\in\widehat{\mathbb{K}_{d-1}}}\in\mathcal{K}_\infty$.

The corresponding convolution formula for functions $f_k\in L^1(SE(d))$ with $k\in\{1,2\}$ at $p>0$ is 
\begin{equation}\label{hatconvp}
\widehat{(f_1\star f_2)}(p)=\widehat{f_2}(p)\widehat{f_1}(p).
\end{equation}

In addition,  the prescription of Plancherel formula (\ref{mainPL.SEd}) then takes the following unified form;
\begin{equation}\label{Radial.SEd}
\int_{SE(d)}|f(g)|^2\dd g=C(d)\int_0^\infty\|\widehat{f}(p)\|_{\rm HS}^2p^{d-1}\dd p,
\end{equation}
if $f\in L^1\cap L^2(SE(d))$. 

\subsection{Harmonic analysis on $\mathbf{\Gamma}\backslash SE(d)$}
Suppose that $\mathbf{\Gamma}$ is a discrete and co-compact subgroup of $SE(d)$ 
equipped with the counting measure as the bi-invariant Haar measure.  Then  the right coset space $\mathbf{\Gamma}\backslash SE(d):=\{\mathbf{\Gamma} g:g\in SE(d)\}$ is a compact transitive right $SE(d)$-space which the Lie group $SE(d)$ acts on it via the right translations.  

The foundations of abstract harmonic analysis on locally compact homogeneous spaces investigated in different directions,  see \cite{FollH,   AGHF.IJM, HR2, HR1, 50} and the list of references therein. 
Suppose $\mathcal{C}_c(SE(d))$ is the space of all continuous functions on $SE(d)$ with compact supports. The function space $\mathcal{C}(\mathbf{\Gamma}\backslash SE(d))$, that is the set of all continuous functions on $\mathbf{\Gamma}\backslash SE(d)$, consists of all functions 
$\widetilde{f}$, where 
$f\in\mathcal{C}_c(SE(d))$ and
\begin{equation}\label{5.1}
\widetilde{f}(\mathbf{\Gamma} g):=\sum_{\gamma\in\mathbf{\Gamma}}f(\gamma\circ g),
\end{equation}
for every $g\in SE(d)$,  see Proposition 2.48 of \cite{FollH}.

Let $\mu$ be a Radon measure on the right coset space $\mathbf{\Gamma}\backslash SE(d)$ and $h\in SE(d)$. The right translation $\mu_h$ of $\mu$ is defined by $\mu_h(E):=\mu(E\circ h)$, for all Borel subsets $E$ of $\mathbf{\Gamma}\backslash SE(d)$, where $E\circ h:=\{\mathbf{\Gamma} g\circ h:\mathbf{\Gamma} g\in E\}$. 
The measure $\mu$ is called $SE(d)$-invariant if $\mu_h=\mu$, for all $h\in SE(d)$.  Since $SE(d)$ is unimodular,  $\mathbf{\Gamma}$ is discrete and the right coset space $\mathbf{\Gamma}\backslash SE(d)$ is compact,  there exists a finite $SE(d)$-invariant measure $\mu$ on the right coset space $\mathbf{\Gamma}\backslash SE(d)$ satisfying the following Weil's formula  
\begin{equation}\label{TH.m}
\int_{\mathbf{\Gamma}\backslash SE(d)}\widetilde{f}(\mathbf{\Gamma} g)\dd\mu(\mathbf{\Gamma} g)=\int_{SE(d)}f(g)\dd g,
\end{equation}
for every $f\in L^1(SE(d))$,  see Theorem 2.49 of \cite{FollH}.  

\section{\bf Non-Abelian Fourier Series on $\mathbf{\Gamma}\backslash SE(d)$}

Throughout this section,  we study harmonic analysis of non-Abelian Fourier series for square integrable functions on the right coset space of discrete and co-compact subgroups in $SE(d)$.  We here assume that $\mathbf{\Gamma}$ is a discrete co-compact subgroup of $SE(d)$ and $\mu$ is the finite $SE(d)$-invariant measure on the right coset space $\mathbf{\Gamma}\backslash SE(d)$ normalized with respect to Weil's formula (\ref{TH.m}). 

The mathematical theory of Fourier analysis on coset spaces of compact groups discussed in \cite{AGHF.MMJ, AGHF.GGD1} and references therein.  In the case of canonical coset spaces of semi-direct product groups with Abelian normal factor, a concrete approach to the relative Fourier analysis studied in \cite{AGHF.JKMS}.  The later theories rely on assumptions about the group which are not as the case for $SE(d)$ and therefore can not be applied for the case of the right coset space $\mathbf{\Gamma}\backslash SE(d)$.

The following result gives a characterization for $L^2$-function space on the right coset space $\mathbf{\Gamma}\backslash SE(d)$. 

\begin{proposition}
{\it Let $f\in L^1(SE(d))$ with $\widetilde{|f|}\in L^2(\mathbf{\Gamma}\backslash SE(d),\mu)$. Then $f\in L^2(SE(d))$ and  
\[
\|f\|_{L^2(SE(d))}\le\|\widetilde{|f|}\|_{L^2(\mathbf{\Gamma}\backslash SE(d),\mu)}.
\] 
}\end{proposition}
\begin{proof}
Let $g\in SE(d)$.  Then  
\[
\sum_{\gamma\in\mathbf{\Gamma}}|f(\gamma\circ g)|^2\le\left(\sum_{\gamma\in\mathbf{\Gamma}}|f(\gamma\circ g)|\right)^2.
\]
So, using Weil's formula, we obtain   
\begin{align*}
\|f\|_{L^2(SE(d))}^2&=\int_{\mathbf{\Gamma}\backslash SE(d)}\sum_{\gamma\in\mathbf{\Gamma}}|f(\gamma\circ g)|^2d\mu(\mathbf{\Gamma} g)\le\int_{\mathbf{\Gamma}\backslash SE(d)}\left(\sum_{\gamma\in\mathbf{\Gamma}}|f(\gamma\circ g)|\right)^2d\mu(\mathbf{\Gamma} g)
\\&=\int_{\mathbf{\Gamma}\backslash SE(d)}\left(\sum_{\gamma\in\mathbf{\Gamma}}|f|(\gamma\circ g)\right)^2d\mu(\mathbf{\Gamma} g)
=\int_{\mathbf{\Gamma}\backslash SE(d)}\widetilde{|f|}(\mathbf{\Gamma} g)^2d\mu(\mathbf{\Gamma} g)=\|\widetilde{|f|}\|^2_{L^2(\mathbf{\Gamma}\backslash SE(d),\mu)},
\end{align*}
which implies that $f\in L^2(SE(d))$. 
\end{proof}

\begin{corollary}
{\it Let $f\in L^1(SE(d))$ be non-negative and $\widetilde{f}\in L^2(\mathbf{\Gamma}\backslash SE(d),\mu)$. Then $f\in L^2(SE(d))$ and  
\[
\|f\|_{L^2(SE(d))}\le\|\widetilde{f}\|_{L^2(\mathbf{\Gamma}\backslash SE(d),\mu)}.
\] 
}\end{corollary}
\begin{remark} Let $f\in L^1(SE(d))$ with $\widetilde{|f|}\in L^2(\mathbf{\Gamma}\backslash SE(d),\mu)$. Then $\widetilde{f}\in L^2(\mathbf{\Gamma}\backslash SE(d),\mu)$ with 
\[
\|\widetilde{f}\|_{L^2(\mathbf{\Gamma}\backslash SE(d),\mu)}\le\|\widetilde{|f|}\|_{L^2(\mathbf{\Gamma}\backslash SE(d),\mu)}.
\] 
\end{remark}

Assume that $K$ is a subset of $SE(d)$ and $\varphi:\mathbf{\Gamma}\backslash SE(d)\to\mathbb{C}$ satisfies  
\begin{equation}\label{QKint}
\int_K|\varphi(\mathbf{\Gamma} g)|\dd g<\infty.
\end{equation}
For every $p>0$,   let $Q^\varphi_K(p):\mathcal{K}_\infty\to\mathcal{K}_\infty$ be the bounded linear operator defined by 
\begin{equation}\label{QK}
Q_K^\varphi(p):=\int_K\varphi(\mathbf{\Gamma} g)U_p(g)\dd g.
\end{equation}
The operator-valued integral (\ref{QK}) is considered in the weak sense.
In details,  for every $\zeta\in\mathcal{K}_\infty$,  the linear operator $Q_K^\varphi(p)\zeta$ is defined by specifying the corresponding $\mathcal{K}_\infty$-inner product with an arbitrary $\zeta'\in \mathcal{K}_\infty$ and the latter is given by the following explicit formula
\[
\langle Q_K^\varphi(p)\zeta,\zeta'\rangle:=\int_K\varphi(\mathbf{\Gamma} g)\langle U_p(g)\zeta,\zeta'\rangle\dd g.
\]
Since $g\mapsto\langle U_p(g)\zeta,\zeta'\rangle$ is a bounded continuous function on $SE(d)$, the integral on the right is the ordinary integral of a function on $K$. 
In addition,  $\langle Q_K^\varphi(p)\zeta,\zeta'\rangle$ depends linearly on $\zeta$ and antilinearly
on $\zeta'$ and
\begin{align*}
\left|\langle  Q_K^\varphi(p)\zeta,\zeta'\rangle\right|\le\|\zeta\|_{\mathcal{K}_\infty}\|\zeta'\|_{\mathcal{K}_\infty} \left(\int_K|\varphi(\mathbf{\Gamma} g)|\dd g\right).
\end{align*}
Therefore, $Q_K^\varphi(p)$ defines a bounded
linear operator on the Hilbert space $\mathcal{K}_\infty$ with the operator norm 
\[
\left\|Q_K^\varphi(p)\right\|\le\int_K|\varphi(\mathbf{\Gamma} g)|\dd g.
\]
\begin{remark} Let $p>0$.  Some canonical possibilities for well-defined 
$Q_K^\varphi(p)$ are as follow;\\
(i) Let $K$ be a compact subset of $SE(d)$.  Then each continuous $\varphi:\mathbf{\Gamma}\backslash SE(d)\to\mathbb{C}$ satisfies the integrability condition (\ref{QKint}).  So,  every $p>0$, defines the bounded linear operator $Q_K^\varphi(p)$ on $\mathcal{K}_\infty$ with the operator norm 
\[
\left\|Q_K^\varphi(p)\right\|\le \int_K|\varphi(\mathbf{\Gamma} g)|\dd g.
\]
(ii) Suppose that $\Omega$ is a fundamental domain of $\mathbf{\Gamma}$ in $SE(d)$ and $\gamma\in\mathbf{\Gamma}$.
Then every $\varphi\in L^1(\mathbf{\Gamma}\backslash SE(d),\mu)$ satisfies the integrability condition (\ref{QKint}) with 
\[
\int_{\gamma\Omega}|\varphi(\mathbf{\Gamma} g)|\dd g=\|\varphi\|_{L^1(\mathbf{\Gamma}\backslash SE(d),\mu)}<\infty.
\]
So,  every $p>0$ defines 
the bounded linear operator $Q_{\gamma\Omega}^\varphi(p)$ on $\mathcal{K}_\infty$ with the operator norm 
\[
\|Q_{\gamma\Omega}^\varphi(p)\|\le \|\varphi\|_{L^1(\mathbf{\Gamma}\backslash SE(d),\mu)}. 
\]
\end{remark}

\begin{lemma}\label{12.Lem}
Suppose $\mathbf{\Gamma}$ is a discrete co-compact subgroup of $SE(d)$ and $\Omega$ is a fundamental domain of $\mathbf{\Gamma}$ in $SE(d)$. 
Assume that $\gamma\in\mathbf{\Gamma}$ and $\varphi\in L^2(\mathbf{\Gamma}\backslash SE(d),\mu)$.  Let $E_{\gamma\Omega}$ be the characteristic function of $\gamma\Omega$ and $\varphi_\gamma(g):=E_{\gamma\Omega}(g)\varphi(\mathbf{\Gamma} g)$ for every $g\in SE(d)$.  Then $\varphi_\gamma\in L^1\cap L^2(SE(d))$.
\end{lemma}
\begin{proof}
Since $\mathbf{\Gamma}\backslash SE(d)$ is compact,  and so it has finite volume,  we conclude that  $L^2(\mathbf{\Gamma}\backslash SE(d),\mu)\subseteq L^2(\mathbf{\Gamma}\backslash SE(d),\mu) $ and $\|\varphi\|_{L^1(\mathbf{\Gamma}\backslash SE(d),\mu)}\le \|\varphi\|_{L^2(\mathbf{\Gamma}\backslash SE(d),\mu)}$.  Let $p\in\{1,2\}$. Then 
\begin{align*}
\|\varphi_\gamma\|_{L^p(SE(d))}^p&=\int_{SE(d)}|\varphi_\gamma(g)|^p\dd g=\int_{SE(d)}E_{\gamma\Omega}(g)|\varphi(\mathbf{\Gamma} g)|^p\dd g
\\&=\int_{SE(d)}E_{\Omega}(\gamma^{-1}g)|\varphi(\mathbf{\Gamma} g)|^p\dd g
=\int_{SE(d)}E_{\Omega}(g)|\varphi(\mathbf{\Gamma} g)|^p\dd g
\\&=\int_{\Omega}|\varphi(\mathbf{\Gamma} g)|^p\dd g=\|\varphi\|_{L^p(\mathbf{\Gamma}\backslash SE(d),\mu)}^p\le\|\varphi\|_{L^2(\mathbf{\Gamma}\backslash SE(d),\mu)}^p,
\end{align*}
implying that $\varphi_\gamma\in L^1\cap L^2(SE(d))$.
\end{proof}

The following theorem presents an explicit formulation for inner-product of functions on the right coset space $\mathbf{\Gamma}\backslash SE(d)$ in terms of the trace of the non-Abelian Fourier integral operators on $SE(d)$.

\begin{theorem}\label{Main}
Suppose $\mathbf{\Gamma}$ is a discrete co-compact subgroup of $SE(d)$ and $\Omega$ is a fundamental domain of $\mathbf{\Gamma}$ in $SE(d)$.
Let $f\in L^1\cap L^2(SE(d))$ such that $\widetilde{f}\in L^2(\mathbf{\Gamma}\backslash SE(d),\mu)$ and $\varphi\in L^2(\mathbf{\Gamma}\backslash SE(d),\mu)$. Then 
\[
\langle\widetilde{f},\varphi\rangle=C(d)\sum_{\gamma\in\mathbf{\Gamma}}\int_0^\infty\mathrm{tr}\left[\widehat{f}(p)Q_{\gamma\Omega}^{\overline{\varphi}}(p)\right]p^{d-1}\dd p,
\]
where 
\[
Q_{\gamma\Omega}^{\overline{\varphi}}(p):=\int_{\gamma\Omega}\overline{\varphi(\mathbf{\Gamma} g)}U_p(g)\dd g,
\]
for every $\gamma\in\mathbf{\Gamma}$ and $p>0$.
\end{theorem}
\begin{proof}
Using Weil's formula (\ref{TH.m}),  we have 
\begin{equation}\label{M0}
\langle\widetilde{f},\varphi\rangle
=\int_{\mathbf{\Gamma}\backslash SE(d)}\widetilde{f}(\mathbf{\Gamma}g)\overline{\varphi(\mathbf{\Gamma}g)}\dd\mu(\mathbf{\Gamma}g)=\int_{SE(d)}f(g)\overline{\varphi(\mathbf{\Gamma} g)}\dd g.
\end{equation}
Since $\Omega$ is a fundamental domain of $\mathbf{\Gamma}$ in $SE(d)$, using (\ref{M0}), we get 
\begin{equation}\label{M1}
\langle\widetilde{f},\varphi\rangle=\sum_{\gamma\in\mathbf{\Gamma}}\int_{\Omega}f(\gamma\omega)\overline{\varphi(\mathbf{\Gamma}\omega)}\dd\omega.
\end{equation}
Assume that $\gamma\in\mathbf{\Gamma}$ is given. Let $\varphi_\gamma:SE(d)\to\mathbb{C}$ be given by $\varphi_\gamma(g):=E_{\gamma\Omega}(g)\varphi(\mathbf{\Gamma} g)$ for every $g\in SE(d)$.
Invoking Lemma \ref{12.Lem}, we have $\varphi_\gamma\in L^1\cap L^2(SE(d))$. Applying Plancherel formula on $SE(d)$,  formula (\ref{Radial.SEd}),  and the polarization identity we obtain  
\begin{equation}\label{M2}
\int_{\Omega}f(\gamma\omega)\overline{\varphi(\mathbf{\Gamma}\omega)}\dd\omega=C(d)\int_0^\infty\mathrm{tr}\left[\widehat{f}(p)\widehat{\varphi_\gamma}(p)^*\right]p^{d-1}\dd p.
\end{equation}
Indeed, we have 
\begin{align*}
\int_{\Omega}f(\gamma\omega)\overline{\varphi(\mathbf{\Gamma}\omega)}\dd\omega&=\int_{\gamma\Omega}f(g)\overline{\varphi(\mathbf{\Gamma} g)}\dd g=\int_{SE(d)}E_{\gamma\Omega}(g)f(g)\overline{\varphi(\mathbf{\Gamma} g)}\dd g\\&=\int_{SE(d)}f(g)\overline{\varphi_\gamma(g)}\dd g=C(d)\int_0^\infty\mathrm{tr}\left[\widehat{f}(p)\widehat{\varphi_\gamma}(p)^*\right]p^{d-1}\dd p.
\end{align*}
Suppose $p\in(0,\infty)$. Then 
\begin{align*}
\widehat{\varphi_\gamma}(p)&=\int_{SE(d)}\varphi_\gamma(g)U_p(g)^*\dd g
=\int_{SE(d)}E_{\gamma\Omega}(g)\varphi(\mathbf{\Gamma} g)U_p(g)^*\dd g
=\int_{\gamma\Omega}\varphi(\mathbf{\Gamma} g)U_p(g)^*\dd g,
\end{align*}
implying that 
\[
\widehat{\varphi_\gamma}(p)^*=\left(\int_{\gamma\Omega}\varphi(\mathbf{\Gamma} g)U_p(g)^*\dd g\right)^*=Q_{\gamma\Omega}^{\overline{\varphi}}(p).
\]
So, using (\ref{M2}), we achieve
\[
\int_{\Omega}f(\gamma\omega)\overline{\varphi(\mathbf{\Gamma} \omega)}\dd\omega=C(d)\int_0^\infty\mathrm{tr}\left[\widehat{f}(p)Q_{\gamma\Omega}^{\overline{\varphi}}(p)\right]p^{d-1}\dd p.
\]
Therefore,  applying (\ref{M1}), we conclude that  
\begin{align*}
\langle\widetilde{f},\varphi\rangle
&=\sum_{\gamma\in\mathbf{\Gamma}}\int_{\Omega}f(\gamma\omega)\overline{\varphi(\mathbf{\Gamma}\omega)}\dd\omega
=C(d)\sum_{\gamma\in\mathbf{\Gamma}}\int_0^\infty\mathrm{tr}\left[\widehat{f}(p)Q_{\gamma\Omega}^{\overline{\varphi}}(p)\right]p^{d-1}\dd p.
\end{align*}
\end{proof}

We then conclude the following formulas for compactly supported functions on $SE(d)$. 

\begin{proposition}\label{main.K}
{\it Suppose $\mathbf{\Gamma}$ is a discrete co-compact subgroup of $SE(d)$ and $\Omega$ is a fundamental domain of $\mathbf{\Gamma}$ in $SE(d)$.  Let $K$ be a compact subset of $SE(d)$ and $\mathbb{J}=\{\gamma_1,\cdots,\gamma_{n}\}\subset\mathbf{\Gamma}$ be a finite subset of $\mathbf{\Gamma}$ such that $K=\bigcup_{k=1}^{n}\gamma_k\Omega$.  Assume that $f:SE(d)\to\mathbb{C}$ is a continuous function supported on $K$ and $\varphi\in L^2(\mathbf{\Gamma}\backslash SE(d),\mu)$.  Then
 \[
\langle\widetilde{f},\varphi\rangle=C(d)\int_0^\infty\mathrm{tr}\left[\widehat{f}(p)Q_K^{\overline{\varphi}}(p)\right]p^{d-1}\dd p.
\]   
}\end{proposition}
\begin{proof}
Let $\gamma\in\mathbf{\Gamma}$ such that $\gamma\not\in\mathbb{J}$. Then 
\[
\int_0^\infty\mathrm{tr}\left[\widehat{f}(p)Q_{\gamma\Omega}^{\overline{\varphi}}(p)\right]p^{d-1}\dd p=0.
\]
So,  using Theorem \ref{Main},  we achieve 
\begin{align*}
\langle\widetilde{f},\varphi\rangle
&=C(d)\sum_{\gamma\in\mathbf{\Gamma}}\int_0^\infty\mathrm{tr}\left[\widehat{f}(p)Q_{\gamma\Omega}^{\overline{\varphi}}(p)\right]p^{d-1}\dd p
\\&=C(d)\sum_{k=1}^n\int_0^\infty\mathrm{tr}\left[\widehat{f}(p)Q_{\gamma_k\Omega}^{\overline{\varphi}}(p)\right]p^{d-1}\dd p
\\&=C(d)\int_0^\infty\sum_{k=1}^n\mathrm{tr}\left[\widehat{f}(p)Q_{\gamma_k\Omega}^{\overline{\varphi}}(p)\right]p^{d-1}\dd p
\\&=C(d)\int_0^\infty\mathrm{tr}\left[\widehat{f}(p)\left(\sum_{k=1}^nQ_{\gamma_k\Omega}^{\overline{\varphi}}(p)\right)\right]p^{d-1}\dd p
=C(d)\int_0^\infty\mathrm{tr}\left[\widehat{f}(p)Q_K^{\overline{\varphi}}(p)\right]p^{d-1}\dd p.
\end{align*}
\end{proof}

\begin{corollary}\label{main.O}
{\it Suppose $\mathbf{\Gamma}$ is a discrete co-compact subgroup of $SE(d)$ and $\Omega$ is a fundamental domain of $\mathbf{\Gamma}$ in $SE(d)$.
Let $f:SE(d)\to\mathbb{C}$ be a continuous function supported in $\Omega$ and $\varphi\in L^2(\mathbf{\Gamma}\backslash SE(d),\mu)$. Then 
\[
\langle\widetilde{f},\varphi\rangle=C(d)\int_0^\infty\mathrm{tr}\left[\widehat{f}(p)Q_{\Omega}^{\overline{\varphi}}(p)\right]p^{d-1}\dd p.
\]
}\end{corollary}

\begin{remark}  Proposition \ref{main.K} and Corollary \ref{main.O} appeared in \cite{AGHF.GSC.PAMQ},  for the case $d=2$ and $f$ is compactly supported satisfying the additional condition $\widehat{f}\in\mathcal{H}^1(0,\infty)$.
\end{remark}

We then conclude  the following explicit formulation for inner-product of functions on the right coset space $\mathbf{\Gamma}\backslash SE(d)$ in terms of the matrix elements of the non-Abelian Fourier integral operator on $SE(d)$.

\begin{theorem}\label{Main.Mat}
Let $f\in L^1(SE(d))$ such that $\widetilde{f}\in L^2(\mathbf{\Gamma}\backslash SE(d),\mu)$.  Suppose $\varphi\in L^2(\mathbf{\Gamma}\backslash SE(d),\mu)$. Then 
\[
\langle\widetilde{f},\varphi\rangle=C(d)\sum_{\gamma\in\mathbf{\Gamma}}\sum_{n=-\infty}^\infty\sum_{m=-\infty}^\infty\int_0^\infty \widehat{f}(p)_{m,n}Q_{\gamma\Omega}^{\overline{\varphi}}(p)_{n,m}p^{d-1}\dd p,
\]
where for every $\gamma\in\mathbf{\Gamma}$ we have 
\[
\sum_{n=-\infty}^\infty\sum_{m=-\infty}^\infty\int_0^\infty|\widehat{f}(p)_{m,n}||Q_{\gamma\Omega}^{\overline{\varphi}}(p)_{n,m}|p^{d-1}\dd p<\infty.
\]
\end{theorem}
\begin{proof}
Suppose that $\gamma\in\mathbf{\Gamma}$.  Then 
\begin{align*}
\int_0^\infty\mathrm{tr}\left[\widehat{f}(p)Q_{\gamma\Omega}^{\overline{\varphi}}(p)\right]p^{d-1}\dd p
&=\int_0^\infty\sum_{n=-\infty}^\infty\langle\widehat{f}(p)Q_{\gamma\Omega}^{\overline{\varphi}}(p)\mathbf{e}_n,\mathbf{e}_n\rangle p^{d-1}\dd p
\\&=\int_0^\infty\sum_{n=-\infty}^\infty\langle Q_{\gamma\Omega}^{\overline{\varphi}}(p)\mathbf{e}_n,\widehat{f}(p)^*\mathbf{e}_n\rangle p^{d-1}\dd p
\\&=\int_0^\infty\sum_{n=-\infty}^\infty\sum_{m=-\infty}^\infty\langle Q_{\gamma\Omega}^{\overline{\varphi}}(p)\mathbf{e}_n,\mathbf{e}_m\rangle\langle\mathbf{e}_m,\widehat{f}(p)^*\mathbf{e}_n\rangle p^{d-1}\dd p
\\&=\int_0^\infty\sum_{n=-\infty}^\infty\sum_{m=-\infty}^\infty \widehat{f}(p)_{m,n}Q_{\gamma\Omega}^{\overline{\varphi}}(p)_{n,m}p^{d-1}\dd p.
\end{align*}
Using Cauchy–Schwarz inequality,  we have 
\[
\int_0^\infty|\widehat{f}(p)_{m,n}||Q_{\gamma\Omega}^{\overline{\varphi}}(p)_{n,m}|p^{d-1}\dd p\le\left(\int_0^\infty|\widehat{f}(p)_{m,n}|^2p^{d-1}\dd p\right)^{1/2}\left(\int_0^\infty|Q_{\gamma\Omega}^{\overline{\varphi}}(p)_{n,m}|^2p^{d-1}\dd p\right)^{1/2},
\]
for every $m,n\in\mathbb{Z}$.  In addition,  we get 
\begin{align*}
&\sum_{m=-\infty}^\infty\sum_{n=-\infty}^\infty\int_0^\infty|\widehat{f}(p)_{m,n}||Q_{\gamma\Omega}^{\overline{\varphi}}(p)_{n,m}|p^{d-1}\dd p
\\&\le\sum_{m=-\infty}^\infty\sum_{n=-\infty}^\infty\left(\int_0^\infty|\widehat{f}(p)_{m,n}|^2p^{d-1}\dd p\right)^{1/2}\left(\int_0^\infty|Q_{\gamma\Omega}^{\overline{\varphi}}(p)_{n,m}|^2p^{d-1}\dd p\right)^{1/2}
\\&\le\sum_{m=-\infty}^\infty\left(\sum_{n=-\infty}^\infty\int_0^\infty|\widehat{f}(p)_{m,n}|^2p^{d-1}\dd p\right)^{1/2}\left(\sum_{n=-\infty}^\infty\int_0^\infty|Q_{\gamma\Omega}^{\overline{\varphi}}(p)_{n,m}|^2p^{d-1}\dd p\right)^{1/2}
\\&\le\left(\sum_{m=-\infty}^\infty\sum_{n=-\infty}^\infty\int_0^\infty|\widehat{f}(p)_{m,n}|^2p^{d-1}\dd p\right)^{1/2}\left(\sum_{m=-\infty}^\infty\sum_{n=-\infty}^\infty\int_0^\infty|Q_{\gamma\Omega}^{\overline{\varphi}}(p)_{n,m}|^2p^{d-1}\dd p\right)^{1/2}<\infty,
\end{align*}
implying that 
\[
\sum_{m=-\infty}^\infty\sum_{n=-\infty}^\infty\int_0^\infty|\widehat{f}(p)_{m,n}||Q_{\gamma\Omega}^{\overline{\varphi}}(p)_{n,m}|p^{d-1}\dd p<\infty.
\]
Then Fubini's Theorem guarantees that  
\begin{align*}
\int_0^\infty\mathrm{tr}\left[\widehat{f}(p)Q_{\gamma\Omega}^{\overline{\varphi}}(p)\right]p^{d-1}\dd p
&=\int_0^\infty\left(\sum_{n=-\infty}^\infty\sum_{m=-\infty}^\infty \widehat{f}(p)_{m,n}Q_{\gamma\Omega}^{\overline{\varphi}}(p)_{n,m}\right)p^{d-1}\dd p\\&=\sum_{n=-\infty}^\infty\sum_{m=-\infty}^\infty\left(\int_0^\infty \widehat{f}(p)_{m,n}Q_{\gamma\Omega}^{\overline{\varphi}}(p)_{n,m}p^{d-1}\dd p\right).
\end{align*}
Therefore,  using Theorem \ref{Main},  we obtain  
\begin{align*}
\langle\widetilde{f},\varphi\rangle=C(d)\sum_{\gamma\in\mathbf{\Gamma}}\sum_{n=-\infty}^\infty\sum_{m=-\infty}^\infty\int_0^\infty \widehat{f}(p)_{m,n}Q_{\gamma\Omega}^{\overline{\varphi}}(p)_{n,m}p^{d-1}\dd p.
\end{align*}
\end{proof}

\begin{corollary}
{\it Suppose $\mathbf{\Gamma}$ is a discrete co-compact subgroup of $SE(d)$ and $\Omega$ is a fundamental domain of $\mathbf{\Gamma}$ in $SE(d)$.  Let $K$ be a compact subset of $SE(d)$ and $\mathbb{J}=\{\gamma_1,\cdots,\gamma_{n}\}\subset\mathbf{\Gamma}$ be a finite subset of $\mathbf{\Gamma}$ such that $K=\bigcup_{k=1}^{n}\gamma_k\Omega$.  Assume that $f:SE(d)\to\mathbb{C}$ is a continuous function supported on $K$ and $\varphi\in L^2(\mathbf{\Gamma}\backslash SE(d),\mu)$. 
Then 
\[
\langle\widetilde{f},\varphi\rangle=C(d)\sum_{n=-\infty}^\infty\sum_{m=-\infty}^\infty\int_0^\infty \widehat{f}(p)_{m,n}Q_{K}^{\overline{\varphi}}(p)_{n,m}p^{d-1}\dd p.
\]
}\end{corollary}

We finish this section by discussing the following constructive non-Abelian  Fourier type approximation for $L^2$-functions on the right coset space $\mathbf{\Gamma}\backslash SE(d)$ with respect to a given orthonormal basis.

\begin{theorem}\label{MainRecTh0.L2.K}
{\it Suppose that $\mathbf{\Gamma}$ is a discrete and co-compact  subgroup of $SE(d)$ and $\mathcal{E}(\mathbf{\Gamma}):=(\psi_\ell)_{\ell\in\mathbb{I}}$ is an orthonormal basis for the Hilbert function space $L^2(\mathbf{\Gamma}\backslash SE(d),\mu)$.  
Let $f\in L^1\cap L^2(SE(d))$ such that $\widetilde{f}\in L^2(\mathbf{\Gamma}\backslash SE(d),\mu)$. Then  
\begin{equation}\label{FourierInvQ0.L1L2.K}
\widetilde{f}=C(d)\sum_{\ell\in\mathbb{I}}\left(\sum_{\gamma\in\mathbf{\Gamma}}\int_0^\infty\mathrm{tr}\left[\widehat{f}(p)Q_{\gamma\Omega}^\ell(p)\right]p^{d-1}\dd p\right)\psi_\ell,
\end{equation}
where 
\[
Q^\ell_{\gamma\Omega}(p):=\int_{\gamma\Omega}\overline{\psi_\ell(\mathbf{\Gamma} g)} U_p(g)\dd g,
\]
for $p>0$ and $\ell\in\mathbb{I}$.
}\end{theorem}
\begin{proof}
Since $(\psi_\ell)_{\ell\in\mathbb{I}}$ is an orthonormal basis for the Hilbert space $L^2(\mathbf{\Gamma}\backslash SE(d),\mu)$,  we get
\begin{equation}
\widetilde{f}=\sum_{\ell\in\mathbb{I}}\langle\widetilde{f},\psi_\ell\rangle\psi_\ell.
\end{equation} 
Applying Theorem \ref{Main}, we have 
\[
\langle\widetilde{f},\psi_\ell\rangle=C(d)\sum_{\gamma\in\mathbf{\Gamma}}\int_0^\infty\mathrm{tr}\left[\widehat{f}(p)Q_{\gamma\Omega}^\ell(p)\right]p^{d-1}\dd p,
\]
which implies that  
\begin{align*}
\widetilde{f}&=\sum_{\ell\in\mathbb{I}}\langle\widetilde{f},\psi_\ell\rangle\psi_\ell
=C(d)\sum_{\ell\in\mathbb{I}}\left(\sum_{\gamma\in\mathbf{\Gamma}}\int_0^\infty\mathrm{tr}\left[\widehat{f}(p)Q_{\gamma\Omega}^\ell(p)\right]p^{d-1}\dd p\right)\psi_\ell.
\end{align*}
\end{proof}

\begin{corollary}
{\it Suppose $\mathbf{\Gamma}$ is a discrete co-compact subgroup of $SE(d)$ and $\Omega$ is a fundamental domain of $\mathbf{\Gamma}$ in $SE(d)$.  Let $\mathcal{E}(\mathbf{\Gamma}):=(\psi_\ell)_{\ell\in\mathbb{I}}$ be an orthonormal basis for the Hilbert function space $L^2(\mathbf{\Gamma}\backslash SE(d),\mu)$.  Let $K$ be a compact subset of $SE(d)$ and $\mathbb{J}=\{\gamma_1,\cdots,\gamma_{n}\}\subset\mathbf{\Gamma}$ be a finite subset of $\mathbf{\Gamma}$ such that $K=\bigcup_{k=1}^{n}\gamma_k\Omega$.  Assume that $f:SE(d)\to\mathbb{C}$ is a continuous function supported on $K$.  Then 
\begin{equation}
\widetilde{f}=C(d)\sum_{\ell\in\mathbb{I}}\left(\int_0^\infty\mathrm{tr}\left[\widehat{f}(p)Q_{K}^\ell(p)\right]p^{d-1}dp\right)\psi_\ell,
\end{equation}
where 
\[
Q^\ell_{K}(p):=\int_{K}\overline{\psi_\ell(\mathbf{\Gamma} g)} U_p(g)\dd g,
\]
for $p>0$ and $\ell\in\mathbb{I}$.
}\end{corollary}

\section{\bf Absolutely Convergent Non-Abelian Fourier Series on $\mathbf{\Gamma}\backslash SE(d)$}
Suppose that $\mathbf{\Gamma}$ is a discrete co-compact  subgroup of $SE(d)$ and $\mu$ is the finite $SE(d)$-invariant measure on the right coset space $\mathbf{\Gamma}\backslash SE(d)$ which is normalized with respect to Weil's formula (\ref{TH.m}).  Let $\mathcal{E}(\mathbf{\Gamma}):=(\psi_\ell)_{\ell\in\mathbb{I}}$ be 
an orthonormal basis of bounded functions for the Hilbert function space $L^2(\mathbf{\Gamma}\backslash SE(d),\mu)$.  We here discus general theory of absolutely convergent non-Abelian Fourier series on $\mathbf{\Gamma}\backslash SE(d)$.

Assume that $\mathcal{A}(\mathcal{E})$ is the linear subspace of $L^2(\mathbf{\Gamma}\backslash SE(d),\mu)$ given by 
\[
\mathcal{A}(\mathcal{E}):=\left\{\psi\in L^2(\mathbf{\Gamma}\backslash SE(d),\mu):
\sum_{\ell\in\mathbb{I}}|\langle\psi,\psi_\ell\rangle|\|\psi_\ell\|_{\infty}<\infty\right\}.
\]

The following observations lists basic properties of function space $\mathcal{A}(\mathcal{E})$.

\begin{proposition}
{\it Let $\mathbf{\Gamma}$ be a discrete co-compact  subgroup of $SE(d)$ and $\mathcal{E}(\mathbf{\Gamma}):=(\psi_\ell)_{\ell\in\mathbb{I}}$ be an orthonormal basis of bounded functions for the Hilbert function space $L^2(\mathbf{\Gamma}\backslash SE(d),\mu)$. Then, 
\begin{enumerate}
\item The sequence $(\|\psi_\ell\|_{\infty}^{-1})_{\ell\in\mathbb{I}}$ is bounded. 
\item If $(a_\ell)\in\bfell^{1}(\mathbb{I})$ then $(\|\psi_\ell\|_{\infty}^{-1}a_\ell)\in\bfell^2(\mathbb{I})$.
\item If $(a_\ell)\in\bfell^{1}(\mathbb{I})$ then  
\[
\sum_{\ell\in\mathbb{I}}\|\psi_\ell\|_{\infty}^{-1}a_\ell\psi_\ell\in\mathcal{A}(\mathcal{E}).
\]
\end{enumerate}
}\end{proposition}
\begin{proof}
(1) Suppose $\ell\in\mathbb{I}$ is given.  
Since $\mathbf{\Gamma}\backslash SE(d)$ is compact and so has finite volume,  we have  
\[
\|\psi_\ell\|_{L^2(\mathbf{\Gamma}\backslash SE(d),\mu)}^2=\int_{L^2(\mathbf{\Gamma}\backslash SE(d),\mu)}|\psi_\ell(\mathbf{\Gamma} g)|^2d\mu(\mathbf{\Gamma} g)\le\|\psi_\ell\|_{\infty}^2\mu(\mathbf{\Gamma}\backslash SE(d)).
\]
Henceforth,  we obtain 
\[
\|\psi_\ell\|_{\infty}^{-1}\le\|\psi_\ell\|_{L^2(\mathbf{\Gamma}\backslash SE(d),\mu)}^{-1}\sqrt{\mu(\mathbf{\Gamma}\backslash SE(d))}=\mu(\mathbf{\Gamma}\backslash SE(d))^{1/2}.
\] 
(2) Let $(a_\ell)\in\bfell^{1}(\mathbb{I})$ be given.
Invoking boundedness of the sequence $(\|\psi_\ell\|_{\infty}^{-1})_{\ell\in\mathbb{I}}$,  we conclude that 
$(\|\psi_\ell\|_{\infty}^{-1}a_\ell)\in\bfell^{1}(\mathbb{I})$.
Since $\bfell^1(\mathbb{I})\subseteq\bfell^2(\mathbb{I})$,  we get $(\|\psi_\ell\|_{\sup}^{-1}a_\ell)\in\bfell^2(\mathbb{I})$.\\
(3) Assume that $(a_\ell)\in\bfell^{1}(\mathbb{I})$ is given.  Using (2), we have $(\|\psi_\ell\|_{\infty}^{-1}a_\ell)\in\bfell^2(\mathbb{I})$. So,  we get 
$\sum_{\ell\in\mathbb{I}}\|\psi_\ell\|_{\infty}^{-1}a_\ell\psi_\ell\in L^2(\mathbf{\Gamma}\backslash SE(d),\mu)$.  Suppose $\psi:=\sum_{\ell\in\mathbb{I}}\|\psi_\ell\|_{\infty}^{-1}a_\ell\psi_\ell$ and $\ell'\in\mathbb{I}$.
Then $\langle\psi,\psi_{\ell'}\rangle=\|\psi_{\ell'}\|_{\infty}^{-1}a_{\ell'}$. Hence,  we achieve 
\[
\sum_{\ell'\in\mathbb{I}}|\langle\psi,\psi_{\ell'}\rangle|\|\psi_{\ell'}\|_{\infty}=\sum_{\ell'\in\mathbb{I}}|a_{\ell'}|<\infty,
\]
which implies that $\sum_{\ell\in\mathbb{I}}\|\psi_\ell\|_{\infty}^{-1}a_\ell\psi_\ell\in\mathcal{A}(\mathcal{E})$.
\end{proof}

\begin{theorem}\label{Main.A.E}
{\it Let $\mathbf{\Gamma}$ be a discrete co-compact  subgroup of $SE(d)$ and $\mathcal{E}(\mathbf{\Gamma}):=(\psi_\ell)_{\ell\in\mathbb{I}}$ be an orthonormal basis
of bounded functions for the Hilbert function space $L^2(\mathbf{\Gamma}\backslash SE(d),\mu)$.  Suppose $\psi\in\mathcal{A}(\mathcal{E})$ is given.  Then, 
\begin{enumerate}
\item The series $\sum_{\ell\in\mathbb{I}}\langle\psi,\psi_\ell\rangle\psi_\ell$ converges uniformly on $\mathbf{\Gamma}\backslash SE(d)$. 
\item For $\mu$-a.e. $\mathbf{\Gamma} g\in\mathbf{\Gamma}\backslash SE(d)$, we have 
\begin{equation}\label{a.e.main}
\psi(\mathbf{\Gamma} g)=\sum_{\ell\in\mathbb{I}}\langle\psi,\psi_\ell\rangle\psi_\ell(\mathbf{\Gamma} g).
\end{equation}
\end{enumerate}
}\end{theorem}
\begin{proof}
(1) Applying the Weierstrass $M$-test,  we conclude that the series $\sum_{\ell\in\mathbb{I}}\langle\psi,\psi_\ell\rangle\psi_\ell$ converges uniformly on $\mathbf{\Gamma}\backslash SE(d)$.\\
(2) Invoking (1), the complex series $\sum_{\ell\in\mathbb{I}}\langle\psi,\psi_\ell\rangle\psi_\ell(\mathbf{\Gamma} g)$ converges,  for every $g\in SE(d)$. 
Let $\varphi(\mathbf{\Gamma} g):=\sum_{\ell\in\mathbb{I}}\langle\psi,\psi_\ell\rangle\psi_\ell(\mathbf{\Gamma} g)$. 
Then, $\varphi:\mathbf{\Gamma}\backslash SE(d)\to\mathbb{C}$  is a well-defined complex valued bounded function.  In addition, using Minkowski's integral inequality for integrals,  we have 
\begin{align*}
\|\varphi\|_{L^2(\mathbf{\Gamma}\backslash SE(d),\mu)}&=\left(\int_{\mathbf{\Gamma}\backslash SE(d)}|\varphi(\mathbf{\Gamma} g)|^2\dd\mu(\mathbf{\Gamma} g)\right)^{1/2}
=\left(\int_{\mathbf{\Gamma}\backslash SE(d)}\left|\sum_{\ell\in\mathbb{I}}\langle\psi,\psi_\ell\rangle\psi_\ell(\mathbf{\Gamma} g)\right|^2\dd\mu(\mathbf{\Gamma} g)\right)^{1/2}
\\&\le\sum_{\ell\in\mathbb{I}}\left(\int_{\mathbf{\Gamma}\backslash SE(d)}|\langle\psi,\psi_\ell\rangle|^2|\psi_\ell(\mathbf{\Gamma} g)|^2\dd\mu(\mathbf{\Gamma} g)\right)^{1/2}=\sum_{\ell\in\mathbb{I}}|\langle\psi,\psi_\ell\rangle|\|\psi_\ell\|_{L^2(\mathbf{\Gamma}\backslash SE(d),\mu)}<\infty,
\end{align*}
which guarantees that $\varphi\in L^2(\mathbf{\Gamma}\backslash SE(d),\mu)$.
Let $\ell'\in\mathbb{I}$ be arbitrary.  Then  
\begin{align*}
\langle\varphi,\psi_{\ell'}\rangle&=\int_{\mathbf{\Gamma}\backslash SE(d)}\varphi(\mathbf{\Gamma} g)\overline{\psi_{\ell'}(\mathbf{\Gamma} g)}d\mu(\mathbf{\Gamma} g)
\\&=\int_{\mathbf{\Gamma}\backslash SE(d)}\left(\sum_{\ell\in\mathbb{I}}\langle\psi,\psi_\ell\rangle\psi_\ell(\mathbf{\Gamma} g)\right)\overline{\psi_{\ell'}(\mathbf{\Gamma} g)}d\mu(\mathbf{\Gamma} g)
\\&=\sum_{\ell\in\mathbb{I}}\langle\psi,\psi_\ell\rangle\left(\int_{\mathbf{\Gamma}\backslash SE(d)}\psi_\ell(\mathbf{\Gamma} g)\overline{\psi_{\ell'}(\mathbf{\Gamma} g)}d\mu(\mathbf{\Gamma} g)\right)\\&=\sum_{\ell\in\mathbb{I}}\langle\psi,\psi_\ell\rangle\langle\psi_\ell,\psi_{\ell'}\rangle
=\sum_{\ell\in\mathbb{I}}\langle\psi,\psi_\ell\rangle\delta_{\ell,\ell'}=\langle\psi,\psi_{\ell'}\rangle,
\end{align*}
implying that $\varphi=\psi$ in $L^2(\mathbf{\Gamma}\backslash SE(d),\mu)$. 
Therefore,  Equation (\ref{a.e.main}) holds for $\mu$-a.e. $\mathbf{\Gamma} g\in\mathbf{\Gamma}\backslash SE(d)$.\\
\end{proof}

\begin{corollary}\label{main.cts}
{\it Let $\mathbf{\Gamma}$ be a discrete co-compact  subgroup of $SE(d)$ and $\mathcal{E}(\mathbf{\Gamma}):=(\psi_\ell)_{\ell\in\mathbb{I}}$ be an orthonormal basis of continuous functions for the Hilbert function space $L^2(\mathbf{\Gamma}\backslash SE(d),\mu)$. 
Suppose $\psi\in\mathcal{A}(\mathcal{E})$ is continuous and $\mathbf{\Gamma} g\in\mathbf{\Gamma}\backslash SE(d)$ is arbitrary.  Then, 
\[
\psi(\mathbf{\Gamma} g)=\sum_{\ell\in\mathbb{I}}\langle\psi,\psi_\ell\rangle\psi_\ell(\mathbf{\Gamma} g).
\]
}\end{corollary}

Then we conclude the following pointwise approximation for functions on the right coset space $\mathbf{\Gamma}\backslash SE(d)$ using non-Abelian Fourier transform on $SE(d)$.

\begin{theorem}\label{ABC.Main}
Suppose $\mathbf{\Gamma}$ is a discrete co-compact subgroup of $SE(d)$ and $\mathcal{E}(\mathbf{\Gamma}):=(\psi_\ell)_{\ell\in\mathbb{I}}$ be an orthonormal basis of bounded functions for the Hilbert function space $L^2(\mathbf{\Gamma}\backslash SE(d),\mu)$.  
Let $f\in L^1\cap L^2(SE(d))$ with $\widetilde{f}\in L^2(\mathbf{\Gamma}\backslash SE(d),\mu)$ and $\widetilde{f}\in\mathcal{A}(\mathcal{E})$.  Then,  for $\mu$-a.e.  $\mathbf{\Gamma} h\in \mathbf{\Gamma}\backslash SE(d)$, we have 
\begin{equation}\label{FourierInvQ0.K}
\widetilde{f}(\mathbf{\Gamma} h)=C(d)\sum_{\ell\in\mathbb{I}}\left(\sum_{\gamma\in\mathbf{\Gamma}}\int_0^\infty\mathrm{tr}\left[\widehat{f}(p)Q_{\gamma\Omega}^\ell(p)\right]p^{d-1}\dd p\right)\psi_\ell(\mathbf{\Gamma} h),
\end{equation}
\end{theorem}
\begin{proof}
Using Theorem \ref{Main.A.E}(2), we get 
\begin{equation}\label{cldec}
\widetilde{f}(\mathbf{\Gamma} h)=\sum_{\ell\in\mathbb{I}}\langle\widetilde{f},\psi_\ell\rangle\psi_\ell(\mathbf{\Gamma} h),
\end{equation}
for $\mu$-a.e. $\mathbf{\Gamma} h\in \mathbf{\Gamma}\backslash SE(d)$.  In addition,  using Theorem \ref{Main}, we obtain   
\begin{equation}\label{cl.0}
\langle\widetilde{f},\psi_\ell\rangle=C(d)\sum_{\gamma\in\mathbf{\Gamma}}\int_0^\infty\mathrm{tr}\left[\widehat{f}(p)Q_{\gamma\Omega}^\ell(p)\right]p^{d-1}\dd p.
\end{equation}
for all $\ell\in\mathbb{I}$.   Then applying (\ref{cl.0}) in (\ref{cldec}) we get 
\begin{align*}
\widetilde{f}(\mathbf{\Gamma} h)&=\sum_{\ell\in\mathbb{I}}\langle\widetilde{f},\psi_\ell\rangle\psi_\ell(\mathbf{\Gamma} h)
=C(d)\sum_{\ell\in\mathbb{I}}\left(\int_0^\infty\mathrm{tr}\left[\widehat{f}(p)Q_K^\ell(p)\right]p^{d-1}\dd p\right)\psi_\ell(\mathbf{\Gamma} h),
\end{align*}
for $\mu$-a.e. $\mathbf{\Gamma} h\in \mathbf{\Gamma}\backslash SE(d)$. 
\end{proof}

\begin{proposition}
{\it Suppose $\mathbf{\Gamma}$ is a discrete co-compact  subgroup of $SE(d)$ and $\mathcal{E}(\mathbf{\Gamma}):=(\psi_\ell)_{\ell\in\mathbb{I}}$ is an orthonormal basis of continuous functions for the Hilbert function space $L^2(\mathbf{\Gamma}\backslash SE(d),\mu)$.  
Let $K\subset SE(d)$ is compact and $f\in\mathcal{C}_c(SE(d))$ with $\mathrm{supp}(f)\subset K$ and $\widetilde{f}\in\mathcal{A}(\mathcal{E})$.  Then,  for every $\mathbf{\Gamma} h\in \mathbf{\Gamma}\backslash SE(d)$, we have 
\begin{equation}
\widetilde{f}(\mathbf{\Gamma} h)=C(d)\sum_{\ell\in\mathbb{I}}\left(\int_0^\infty\mathrm{tr}\left[\widehat{f}(p)Q_{K}^\ell(p)\right]p^{d-1}\dd p\right)\psi_\ell(\mathbf{\Gamma} h),
\end{equation}
}\end{proposition}
\begin{proof}
Since $f\in\mathcal{C}_c(SE(d))$,  we conclude that 
$\widetilde{f}\in\mathcal{C}(\mathbf{\Gamma}\backslash SE(d))$.  So,  using Corollary \ref{main.cts}, we get that the function $\widetilde{f}$,  satisfies Equation (\ref{FourierInvQ0.K}), for every $\mathbf{\Gamma} h\in \mathbf{\Gamma}\backslash SE(d)$.
\end{proof}

\begin{corollary}\label{MainRecTh0}
{\it Suppose $\mathbf{\Gamma}$ is a discrete co-compact subgroup of $SE(d)$ and $\mathcal{E}(\mathbf{\Gamma}):=(\psi_\ell)_{\ell\in\mathbb{I}}$ is an orthonormal basis of continuous functions for the Hilbert function space $L^2(\mathbf{\Gamma}\backslash SE(d),\mu)$.  Assume that $\Omega$ is a fundamental domain of $\mathbf{\Gamma}$ in $SE(d)$.  Let $f:SE(d)\to\mathbb{C}$ be a continuous function supported in $\Omega$ such that $\widetilde{f}\in\mathcal{A}(\mathcal{E})$. Then, for every $\omega\in\Omega$, we have 
\begin{equation}\label{FourierInvQ0}
f(\omega)=C(d)\sum_{\ell\in\mathbb{I}}\left(\int_0^\infty\mathrm{tr}\left[\widehat{f}(p)Q_\Omega^\ell(p)\right]pdp\right)\psi_\ell(\mathbf{\Gamma}\omega),
\end{equation}
 where 
\[
Q^\ell_\Omega(p):=\int_{\Omega}\overline{\psi_\ell(\mathbf{\Gamma} g)}U_p(g)\dd g,
\]
for $p>0$ and $\ell\in\mathbb{I}$.
}\end{corollary}

\section{{\bf Non-Abelian Fourier Series of Convolutions on $\mathbf{\Gamma}\backslash SE(d)$}}

Throughout this section, suppose that $\mathbf{\Gamma}$ is a discrete co-compact  subgroup of $SE(d)$ and $\mu$ is the finite $SE(d)$-invariant measure on the right coset space $\mathbf{\Gamma}\backslash SE(d)$ which is normalized with respect to Weil's formula (\ref{TH.m}). 

The notion of convolution of functions on $SE(d)$ by functions on the right coset space $\mathbf{\Gamma}\backslash SE(d)$ is introduced.  We then discuses 
properties of non-Abelian Fourier series for approximating the convolution functions on the right coset space $\mathbf{\Gamma}\backslash SE(d)$. 
As applications for non-Abelian Fourier series of convolution functions, we conclude the paper by some non-Abelian Plancherel type formulas for functions on the right coset space $\mathbf{\Gamma}\backslash SE(d)$. 

\subsection{Convolution Functions on $\mathbf{\Gamma}\backslash SE(d)$}
The mathematical theory of convolution function algebras on coset spaces of compact subgroups investigated in \cite{AGHF.BBMSS, AGHF.IJM} and references therein.  This theory can be applied for convolution integrals on right coset space of compact subgroups in $SE(d)$ which is not the case for $\mathbf{\Gamma}\backslash SE(d)$ if $\mathbf{\Gamma}$ is not a finite subgroup.  We here extend the notion of convolution for functions on the right coset space of $\mathbf{\Gamma}\backslash SE(d)$.

Suppose $f\in L^1(SE(d))$ and $\psi\in L^1(\mathbf{\Gamma}\backslash SE(d),\mu)$. We then define the convolution of $f$ with $\psi$ as the function $\psi\oslash f:\mathbf{\Gamma}\backslash SE(d)\to\mathbb{C}$ via 
\begin{equation}\label{oslash}
(\psi\oslash f)(\mathbf{\Gamma} g):=\int_{SE(d)} \psi(\mathbf{\Gamma} h)f(h^{-1}\circ g)\dd h,
\end{equation}
for $g\in SE(d)$.

Then, using (\ref{oslash}),  we get 
\[
(\psi\oslash f)(\mathbf{\Gamma} g)=\int_{SE(d)} \psi(\mathbf{\Gamma} g\circ h)f(h^{-1})\dd h.
\] 
So,  we conclude that 
\[
\mathbf{\Gamma} g\mapsto \int_{SE(d)}\psi(\mathbf{\Gamma} h)f(h^{-1}\circ g)\dd h,
\]
is well-defined as a function on the right coset space $\mathbf{\Gamma}\backslash SE(d)$.

Next we discuss analytical aspects of the convolution (\ref{oslash}) and shall show that the Banach function space $L^p(\mathbf{\Gamma}\backslash SE(d),\mu)$ equipped with the module action of $L^1(SE(d))$ defined by (\ref{oslash}) is a Banach function module. 

To this end, we first prove the norm property of the convolution (\ref{oslash}). 

\begin{theorem}\label{conv.norm}
Let $p\ge 1$.  Suppose $f\in L^1(SE(d))$ and $\psi\in L^p(\mathbf{\Gamma}\backslash SE(d),\mu)$. Then, $\psi\oslash f\in L^p(\mathbf{\Gamma}\backslash SE(d),\mu)$ with 
\[
\|\psi\oslash f\|_{L^p(\mathbf{\Gamma}\backslash SE(d),\mu)}\le \|f\|_{L^1(SE(d))}\|\psi\|_{L^p(\mathbf{\Gamma}\backslash SE(d),\mu)}.
\]
\end{theorem}
\begin{proof}
Using Minkowski's integral inequality,  we obtain  
\begin{align*}
\|\psi\oslash f\|_{L^p(\mathbf{\Gamma}\backslash SE(d),\mu)}
&=\left(\int_{\mathbf{\Gamma}\backslash SE(d)}\left|\int_{SE(d)} \psi(\mathbf{\Gamma} g\circ h)f(h^{-1})\dd h\right|^p\dd\mu(\mathbf{\Gamma} g)\right)^{1/p}
\\&\le\int_{SE(d)}\left(\int_{\mathbf{\Gamma}\backslash SE(d)} \left|\psi(\mathbf{\Gamma} g\circ h)f(h^{-1})\right|^p\dd\mu(\mathbf{\Gamma} g)\right)^{1/p}\dd h
\\&=\int_{SE(d)}|f(h^{-1})|\left(\int_{\mathbf{\Gamma}\backslash SE(d)} \left|\psi(\mathbf{\Gamma} g\circ h)\right|^p\dd\mu(\mathbf{\Gamma} g)\right)^{1/p}\dd h.
\end{align*}
Since $\mu$ is $SE(d)$-invariant, we have 
\begin{equation*}
\int_{\mathbf{\Gamma}\backslash SE(d)} \left|\psi(\mathbf{\Gamma} g\circ h)\right|^p\dd\mu(\mathbf{\Gamma} g)=\int_{\mathbf{\Gamma}\backslash SE(d)} \left|\psi(\mathbf{\Gamma} g)\right|^p\dd\mu(\mathbf{\Gamma} g\circ h^{-1})=\int_{\mathbf{\Gamma}\backslash SE(d)} \left|\psi(\mathbf{\Gamma} g)\right|^p\dd\mu(\mathbf{\Gamma} g),
\end{equation*}
for every $h\in SE(d)$. Therefore, we get 
\begin{align*}
\|\psi\oslash f\|_{L^p(\mathbf{\Gamma}\backslash SE(d),\mu)}
&=\int_{SE(d)}|f(h^{-1})|\left(\int_{\mathbf{\Gamma}\backslash SE(d)} \left|\psi(\mathbf{\Gamma} g\circ h)\right|^p\dd\mu(\mathbf{\Gamma} g)\right)^{1/p}\dd h
\\&=\|\psi\|_{L^p(\mathbf{\Gamma}\backslash SE(d),\mu)}\left(\int_{SE(d)}|f(h^{-1})|\dd h\right)=\|f\|_{L^1(SE(d))}\|\psi\|_{L^p(\mathbf{\Gamma}\backslash SE(d),\mu)}.
\end{align*}
\end{proof}

\begin{remark}
Let $p\ge 1$.  Suppose $f\in L^1(SE(d))$ and $\psi\in L^p(\mathbf{\Gamma}\backslash SE(d),\mu)$.  Then Theorem \ref{conv.norm} guarantees that the convolution integral in (\ref{oslash}) converges absolutely for $\mu$-almost every $\mathbf{\Gamma} g\in\mathbf{\Gamma}\backslash SE(d)$. 
\end{remark}

\begin{proposition}
{\it Let $p\ge 1$.  Suppose $f\in L^1(SE(d))$ and $\psi\in L^p(\mathbf{\Gamma}\backslash SE(d),\mu)$.  Then 
\begin{equation}
(\psi\oslash f)(\mathbf{\Gamma} g)=\int_{\mathbf{\Gamma}\backslash SE(d)}\psi(\mathbf{\Gamma} h)\left(\sum_{\gamma\in\mathbf{\Gamma}} f(h^{-1}\circ\gamma^{-1}\circ g)\right)\dd\mu(\mathbf{\Gamma} h).
\end{equation}
}\end{proposition}
\begin{proof}
Let $f\in L^1(SE(d))$ and $\psi\in L^1(\mathbf{\Gamma}\backslash SE(d),\mu)$. Let $g\in SE(d)$. Using Weil's formula, we can write 
\begin{align*}
(\psi\oslash f)(\mathbf{\Gamma} g)
&=\int_{SE(d)} \psi(\mathbf{\Gamma} h)f(h^{-1}g)\dd h
\\&=\int_{\mathbf{\Gamma}\backslash SE(d)}\left(\sum_{\gamma\in\mathbf{\Gamma}}\psi(\mathbf{\Gamma} \gamma\circ h) f((\gamma\circ h)^{-1}\circ g)\right)\dd\mu(\mathbf{\Gamma} h)
\\&=\int_{\mathbf{\Gamma}\backslash SE(d)}\left(\sum_{\gamma\in\mathbf{\Gamma}}\psi(\mathbf{\Gamma} h)f((\gamma\circ h)^{-1}\circ g)\right)\dd\mu(\mathbf{\Gamma} h)
\\&=\int_{\mathbf{\Gamma}\backslash SE(d)}\psi(\mathbf{\Gamma} h)\left(\sum_{\gamma\in\mathbf{\Gamma}} f(h^{-1}\circ\gamma^{-1}\circ g)\right)\dd\mu(\mathbf{\Gamma} h).
\end{align*}
\end{proof}

\begin{proposition}\label{tilde.conv}
{\it Let $f_k\in L^1(SE(d))$ with $k\in\{1,2\}$.  Then  
\begin{equation}
\widetilde{f_1\star f_2}=\widetilde{f_1}\oslash f_2.
\end{equation}
}\end{proposition}
\begin{proof}
Let $g\in SE(d)$ be given.  Then    
\begin{align*}
\widetilde{(f_1\star f_2)}(\mathbf{\Gamma} g)
&=\sum_{\gamma\in\mathbf{\Gamma}}(f_1\star f_2)(\gamma\circ g)
\\&=\sum_{\gamma\in\mathbf{\Gamma}}\left(\int_{SE(d)}f_1(h)f_2(h^{-1}\circ\gamma\circ g)\dd h\right)
\\&=\sum_{\gamma\in\mathbf{\Gamma}}\left(\int_{SE(d)}f_1(\gamma\circ h)f_2(h^{-1}\circ g)\dd h\right)
\\&=\int_{SE(d)}\left(\sum_{\gamma\in\mathbf{\Gamma}}f_1(\gamma\circ h)\right)f_2(h^{-1}\circ g)\dd h
\\&=\int_{SE(d)}\widetilde{f_1}(\mathbf{\Gamma} h)f_2(h^{-1}\circ g)\dd h=(\widetilde{f_1}\oslash f_2)(\mathbf{\Gamma} g).
\end{align*}
\end{proof}
 
 \begin{corollary}
{\it Suppose $p\ge 1$ and $\psi\in L^p(\mathbf{\Gamma}\backslash SE(d),\mu)$. 
Let $f_k\in L^1(SE(d))$ with $k\in\{1,2\}$.  Then 
\[
\psi\oslash(f_1\star f_2)=(\psi\oslash f_1)\oslash f_2.
\] 
} \end{corollary}
\begin{proof}
Since the right cost space $\mathbf{\Gamma}\backslash SE(d)$ is compact and hence has finite volume, we have $L^p(\mathbf{\Gamma}\backslash SE(d),\mu)\subseteq L^1(\mathbf{\Gamma}\backslash SE(d),\mu)$. Let $f\in L^1(SE(d))$ such that $\widetilde{f}=\psi$. Using Proposition \ref{tilde.conv} and associativity of the group convolution integral in $L^1(SE(d))$,  
we have 
\begin{align*}
\psi\oslash(f_1\star f_2)
&=\widetilde{f}\oslash(f_1\star f_2)=\widetilde{f\star(f_1\star f_2)}
\\&=\widetilde{(f\star f_1)\star f_2}=\widetilde{f\star f_1}\oslash f_2=(\psi\oslash f_1)\oslash f_2.
\end{align*}
 \end{proof}

\subsection{Non-Abelian Fourier Series of Convolutions}
We then present some constructive expansions for the coefficients of convolution of $L^1(SE(d))$ on $L^1(\mathbf{\Gamma}\backslash SE(d),\mu)$ in the $L^2$-sense.

\begin{theorem}\label{Main.Coff.Conv.K}
Let $\mathbf{\Gamma}$ be a discrete co-compact subgroup of $SE(d)$ and 
$\Omega$ be a fundamental domain of $\mathbf{\Gamma}$ in $SE(d)$. Suppose 
$\varphi\in L^2(\mathbf{\Gamma}\backslash SE(d),\mu)$.  Let $f_k\in L^1\cap L^2(SE(d))$ with $k\in\{1,2\}$ such that $\widetilde{f_1}\oslash f_2\in L^2(\mathbf{\Gamma}\backslash SE(d),\mu)$.  Then
\begin{equation}\label{CoffConv}
\langle \widetilde{f_1}\oslash f_2,\varphi\rangle=C(d)\sum_{\gamma\in\mathbf{\Gamma}}\int_0^\infty\mathrm{tr}\left[\widehat{f_2}(p)\widehat{f_1}(p)Q_{\gamma\Omega}^{\overline{\varphi}}(p)\right]p^{d-1}\dd p,
\end{equation}
\end{theorem}
\begin{proof}
Let $f:=f_1\star f_2$.  Then $f\in L^1(SE(d))$.  Since $f_2\in L^2(SE(d))$ and using Theorem 2.39 of \cite{FollH}, we get $f\in L^2(SE(d))$ and hence $f\in L^1\cap L^2(SE(d))$.  In addition,  Proposition \ref{tilde.conv} implies that $\widetilde{f}\in L^2(\mathbf{\Gamma}\backslash SE(d),\mu)$.   Therefore,  applying Theorem \ref{Main} for $f$, we obtain 
\begin{align*}
\langle \widetilde{f_1}\oslash f_2,\varphi\rangle&=\langle f,\varphi\rangle
\\&=C(d)\sum_{\gamma\in\mathbf{\Gamma}}\int_0^\infty\mathrm{tr}\left[\widehat{f}(p)Q_{\gamma\Omega}^{\overline{\varphi}}(p)\right]p^{d-1}\dd p
\\&=C(d)\sum_{\gamma\in\mathbf{\Gamma}}\int_0^\infty\mathrm{tr}\left[\widehat{f_1\star f_2}(p)Q_{\gamma\Omega}^{\overline{\varphi}}(p)\right]p^{d-1}\dd p
=C(d)\sum_{\gamma\in\mathbf{\Gamma}}\int_0^\infty\mathrm{tr}\left[\widehat{f_2}(p)\widehat{f_1}(p)Q_{\gamma\Omega}^{\overline{\varphi}}(p)\right]p^{d-1}\dd p.
\end{align*}
\end{proof}

\begin{corollary}
{\it Let $\mathbf{\Gamma}$ be a discrete co-compact subgroup of $SE(d)$. 
Suppose $K$ is a compact subset of $SE(d)$ and 
$\varphi\in L^2(\mathbf{\Gamma}\backslash SE(d))$. 
Let $f_k:SE(d)\to\mathbb{C}$ with $k\in\{1,2\}$ be continuous functions supported in $B_k\subset SE(d)$ such that $B_1\circ B_2\subset K$.  Then 
\begin{equation}
\langle \widetilde{f_1}\oslash f_2,\varphi\rangle=C(d)\int_0^\infty\mathrm{tr}\left[\widehat{f_2}(p)\widehat{f_1}(p)Q_K^{\overline{\varphi}}(p)\right]p^{d-1}\dd p,
\end{equation}
}\end{corollary}

Next we conclude some reconstruction expansions including non-Abelian Fourier coefficients for the convolution of $L^1(SE(d))$ on $L^1(\mathbf{\Gamma}\backslash SE(d),\mu)$ in the $L^2$-sense.

\begin{proposition}
{\it Let $\mathbf{\Gamma}$ be a discrete co-compact  subgroup of $SE(d)$ and $\mathcal{E}(\mathbf{\Gamma}):=(\psi_\ell)_{\ell\in\mathbb{I}}$ be an orthonormal basis for the Hilbert function space $L^2(\mathbf{\Gamma}\backslash SE(d),\mu)$.  Let $f_k\in L^1\cap L^2(SE(d))$ with $k\in\{1,2\}$ such that $\widetilde{f_1}\oslash f_2\in L^2(\mathbf{\Gamma}\backslash SE(d),\mu)$.  Then 
\begin{equation}\label{FourierInvQ0ConvL2.K}
\widetilde{f_1}\oslash f_2=C(d)\sum_{\ell\in\mathbb{I}}\left(\sum_{\gamma\in\mathbf{\Gamma}}\int_0^\infty\mathrm{tr}\left[\widehat{f_2}(p)\widehat{f_1}(p)Q_{\gamma\Omega}^\ell(p)\right]p^{d-1}\dd p\right)\psi_\ell,
\end{equation}
}\end{proposition}

\begin{corollary}
{\it Let $\mathbf{\Gamma}$ be a discrete co-compact  subgroup of $SE(d)$ and $\mathcal{E}(\mathbf{\Gamma}):=(\psi_\ell)_{\ell\in\mathbb{I}}$ be an orthonormal basis for the Hilbert function space $L^2(\mathbf{\Gamma}\backslash SE(d),\mu)$.  Suppose $K$ is a given compact subset of $SE(d)$.  Let $f_k:SE(d)\to\mathbb{C}$ with $k\in\{1,2\}$ be continuous functions supported in $B_k\subset SE(d)$ such that $B_1\circ B_2\subset K$.  Then
\begin{equation}
\widetilde{f_1}\oslash f_2=C(d)\sum_{\ell\in\mathbb{I}}\left(\int_0^\infty\mathrm{tr}\left[\widehat{f_2}(p)\widehat{f_1}(p)Q_K^\ell(p)\right]p^{d-1}\dd p\right)\psi_\ell.
\end{equation}
}\end{corollary}

\subsection{Absolutely Convergent Fourier Series of Convolutions on $\mathbf{\Gamma}\backslash SE(d)$}

We then deduce the following constructive expansions for the convolution of 
$L^1(SE(d))$ on $L^1(\mathbf{\Gamma}\backslash SE(d),\mu)$ in the almost everywhere and pointwise senses. 

\begin{theorem}
Let $\mathbf{\Gamma}$ be a discrete co-compact  subgroup of $SE(d)$ and $\mathcal{E}(\mathbf{\Gamma}):=(\psi_\ell)_{\ell\in\mathbb{I}}$ be an orthonormal basis of bounded functions for the Hilbert function space $L^2(\mathbf{\Gamma}\backslash SE(d),\mu)$.  Let $f_k\in L^1\cap L^2(SE(d))$ with $k\in\{1,2\}$ such that $\widetilde{f_1}\oslash f_2\in\mathcal{A}(\mathcal{E})$.  Then,  for $\mu$-a.e.  $\mathbf{\Gamma} g\in\mathbf{\Gamma}\backslash SE(d)$, we have 
\begin{equation}\label{FourierInvQ0.conv.point}
(\widetilde{f_1}\oslash f_2)(\mathbf{\Gamma} g)=C(d)\sum_{\ell\in\mathbb{I}}\left(\sum_{\gamma\in\mathbf{\Gamma}}\int_0^\infty\mathrm{tr}\left[\widehat{f_2}(p)\widehat{f_1}(p)Q_{\gamma\Omega}^\ell(p)\right]p^{d-1}\dd p\right)\psi_\ell(\mathbf{\Gamma} g),
\end{equation}
\end{theorem}
\begin{proof}
Using Theorem \ref{Main.A.E}(2), we have 
\[
(\widetilde{f_1}\oslash f_2)(\mathbf{\Gamma} g)=\sum_{\ell\in\mathbb{I}}\langle\widetilde{f_1}\oslash f_2,\psi_\ell\rangle\psi_\ell(\mathbf{\Gamma} g),
\]
for $\mu$-a.e.  $\mathbf{\Gamma} g\in\mathbf{\Gamma}\backslash SE(d)$.  Invoking Equation (\ref{CoffConv}), we achieve 
\[
\langle\widetilde{f_1}\oslash f_2,\psi_\ell\rangle=C(d)\sum_{\gamma\in\mathbf{\Gamma}}\int_0^\infty\mathrm{tr}\left[\widehat{f_2}(p)\widehat{f_1}(p)Q_{\gamma\Omega}^\ell(p)\right]p^{d-1}\dd p.
\]
\end{proof}

\begin{proposition}
{\it Let $\mathbf{\Gamma}$ be a discrete co-compact subgroup of $SE(d)$ and $\mathcal{E}(\mathbf{\Gamma}):=(\psi_\ell)_{\ell\in\mathbb{I}}$ be an orthonormal basis
of continues functions for the Hilbert function space $L^2(\mathbf{\Gamma}\backslash SE(d),\mu)$.  Suppose $K$ is a given compact subset of $SE(d)$. Let $f_k:SE(d)\to\mathbb{C}$ with $k\in\{1,2\}$ be continuous functions supported in $B_k\subset SE(d)$ such that $B_1\circ B_2\subset K$ and $\widetilde{f_1}\oslash f_2\in\mathcal{A}(\mathcal{E})$. Then, for every $\mathbf{\Gamma} g\in\mathbf{\Gamma}\backslash SE(d)$, we have 
\begin{equation}
(\widetilde{f_1}\oslash f_2)(\mathbf{\Gamma} g)=C(d)\sum_{\ell\in\mathbb{I}}\left(\int_0^\infty\mathrm{tr}\left[\widehat{f_2}(p)\widehat{f_1}(p)Q_K^\ell(p)\right]p^{d-1}\dd p\right)\psi_\ell(\mathbf{\Gamma} g),
\end{equation}
}\end{proposition}

\begin{corollary}
{\it Let $\mathbf{\Gamma}$ be a discrete co-compact subgroup of $SE(d)$ and $\mathcal{E}(\mathbf{\Gamma}):=(\psi_\ell)_{\ell\in\mathbb{I}}$ be an orthonormal basis of continuous functions for the Hilbert function space $L^2(\mathbf{\Gamma}\backslash SE(d),\mu)$.  Suppose $\Omega$ is a fundamental domain of $\mathbf{\Gamma}$ in $SE(d)$. Let $f_k:SE(d)\to\mathbb{C}$ with $k\in\{1,2\}$ be continuous functions supported in $B_k\subset SE(d)$ such that $B_1\circ B_2\subset\Omega$ and $\widetilde{f_1}\oslash f_2\in\mathcal{A}(\mathcal{E})$. Then, for every $\omega\in\Omega$, we have 
\begin{equation}\label{FourierInvQ0.conv.p.F}
(f_1\star f_2)(\omega)=C(d)\sum_{\ell\in\mathbb{I}}\left(\int_0^\infty\mathrm{tr}\left[\widehat{f_2}(p)\widehat{f_1}(p)Q_\Omega^\ell(p)\right]p^{d-1}\dd p\right)\psi_\ell(\mathbf{\Gamma} \omega),
\end{equation}
}\end{corollary}

\subsection{{\bf Plancherel Formula}}

We conclude this section by presenting Plancherel type formulas associated to the non-Abelian Fourier series on $\mathbf{\Gamma}\backslash SE(d)$.

The following result presents the canonical connection of $L^2$-norms on the right coset space $\mathbf{\Gamma}\backslash SE(d)$ with convolution of functions. 

\begin{proposition}\label{plan.alt}
{\it Let $f\in L^1(SE(d))$ with $\widetilde{f}\in L^2(\mathbf{\Gamma}\backslash SE(d),\mu)$ be given.  Then
\[
\|\widetilde{f}\|_{L^2(\mathbf{\Gamma}\backslash SE(d),\mu)}^2=(\widetilde{f}\oslash f^*)(\mathbf{\Gamma}).
\]
}\end{proposition}
\begin{proof}
Suppose $f\in L^1(SE(d))$ with $\widetilde{f}\in L^2(\mathbf{\Gamma}\backslash SE(d),\mu)$.  Then,  using Weils formula,  we get  
\begin{align*}
\|\widetilde{f}\|_{L^2(\mathbf{\Gamma}\backslash SE(d),\mu)}^2&=\int_{\mathbf{\Gamma}\backslash SE(d)}\widetilde{f}(\mathbf{\Gamma} h)\overline{\widetilde{f}(\mathbf{\Gamma} h)}\dd\mu(\mathbf{\Gamma} h)
\\&=\int_{\mathbf{\Gamma}\backslash SE(d)}\sum_{\gamma\in\mathbf{\Gamma}}\widetilde{f}(\mathbf{\Gamma}\gamma\circ h)\overline{f(\gamma\circ h)}\dd\mu(\mathbf{\Gamma} h)
\\&=\int_{SE(d)}\widetilde{f}(\mathbf{\Gamma} h)\overline{f(h)}\dd h
=\int_{SE(d)}\widetilde{f}(\mathbf{\Gamma} h)f^*(h^{-1})\dd h=(\widetilde{f}\oslash f^*)(\mathbf{\Gamma}).
\end{align*}
\end{proof}

We here finish the section by the following general form of Plancherel formula associated to the non-Abelian Fourier series of continuous functions on $\mathbf{\Gamma}\backslash SE(d)$.

\begin{theorem}\label{plan.general}
Let $\mathbf{\Gamma}$ be a discrete co-compact subgroup of $SE(d)$ and $\mathcal{E}(\mathbf{\Gamma}):=(\psi_\ell)_{\ell\in\mathbb{I}}$ be an orthogonal basis of continues functions for the Hilbert function space $L^2(\mathbf{\Gamma}\backslash SE(d),\mu)$. 
Suppose $K$ is a given compact subset of $SE(d)$. Let $f:SE(d)\to\mathbb{C}$ be a given continuous function supported in $B$ with $B\circ B^{-1}\subset K$ and $\widetilde{f}\oslash f^*\in\mathcal{A}(\mathcal{E})$.  Then  
\begin{equation}\label{FourierInvQ0plan}
\|\widetilde{f}\|_{L^2(\mathbf{\Gamma}\backslash SE(d),\mu)}^2=C(d)\sum_{\ell\in\mathbb{I}}\left(\int_0^\infty\mathrm{tr}\left[|\widehat{f}(p)|^2Q_K^\ell(p)\right]p^{d-1}\dd p\right)\psi_\ell(\mathbf{\Gamma}),
\end{equation}
\end{theorem}
\begin{proof}
Applying Equation (\ref{FourierInvQ0.conv.point}) and Proposition \ref{plan.alt}, we obtain  
\begin{align*}
\|\widetilde{f}\|_{L^2(\mathbf{\Gamma}\backslash SE(d),\mu)}^2
&=(\widetilde{f}\oslash f^*)(\mathbf{\Gamma})
\\&=C(d)\sum_{\ell\in\mathbb{I}}\left(\int_0^\infty\mathrm{tr}\left[\widehat{f^*}(p)\widehat{f}(p)Q_K^\ell(p)\right]p^{d-1}\dd p\right)\psi_\ell(\mathbf{\Gamma})
\\&=C(d)\sum_{\ell\in\mathbb{I}}\left(\int_0^\infty\mathrm{tr}\left[\widehat{f}(p)^*\widehat{f}(p)Q_K^\ell(p)\right]p^{d-1}\dd p\right)\psi_\ell(\mathbf{\Gamma})
\\&=C(d)\sum_{\ell\in\mathbb{I}}\left(\int_0^\infty\mathrm{tr}\left[|\widehat{f}(p)|^2Q_K^\ell(p)\right]p^{d-1}\dd p\right)\psi_\ell(\mathbf{\Gamma}).
\end{align*}
\end{proof}

\begin{corollary}\label{plan.F}
{\it Let $\mathbf{\Gamma}$ be a discrete co-compact subgroup of $SE(d)$ and $\mathcal{E}(\mathbf{\Gamma}):=(\psi_\ell)_{\ell\in\mathbb{I}}$ be an orthogonal basis of continues functions for the Hilbert function space $L^2(\mathbf{\Gamma}\backslash SE(d),\mu)$.  Suppose $\Omega$ is a fundamental domain for $\mathbf{\Gamma}$ in $SE(d)$. Let $f:SE(d)\to\mathbb{C}$ be a given continuous function supported in $B$ with $B\circ B^{-1}\subset\Omega$ and $\widetilde{f}\oslash f^*\in\mathcal{A}(\mathcal{E})$.  Then
\begin{equation}\label{FourierInvQ0plan.F}
\|f\|_{L^2(SE(d))}^2=C(d)\sum_{\ell\in\mathbb{I}}\left(\int_0^\infty\mathrm{tr}\left[|\widehat{f}(p)|^2Q_\Omega^\ell(p)\right]p^{d-1}\dd p\right)\psi_\ell(\mathbf{\Gamma}),
\end{equation}
}\end{corollary}

\section{\bf{Examples and Conclusions}}

In this section, we discuss different analytical aspects of the presented theory in the case of $d=2$ and $d=3$.

\subsection{The case $\mathbf{\Gamma}\backslash SE(2)$}
The 2D special Euclidean motion group, $SE(2)$, is the semi-direct product of $\mathbb{R}^2$
with the 2D special orthogonal group $SO(2)$. That is,
\[
SE(2)=\mathbb{R}^2\rtimes SO(2)=SO(2)\ltimes\mathbb{R}^2.
\] 
We denote elements $g\in SE(2)$ as $g=(\mathbf{x},\mathbf{R})$ where $\mathbf{x}\in\mathbb{R}^2$ and $\mathbf{R}\in SO(2)$. For any $g=(\mathbf{x},\mathbf{R})$ and 
$g'=(\mathbf{x}',\mathbf{R}')\in SE(2)$ the group law is written as 
\[
g\circ g'=(\mathbf{x}+\mathbf{R}\mathbf{x}',\mathbf{RR}'),
\]
and 
\[
g^{-1}=(-\mathbf{R}^T\mathbf{x},\mathbf{R}^T)=(-\mathbf{R}^{-1}\mathbf{x},\mathbf{R}^{-1}).
\]

The special Euclidean group $SE(2)$, can be viewed as the set of matrices of the form 
\begin{equation}
g(x_1,x_2,\theta):=\left(\begin{array}{ccc}
\cos\theta & -\sin\theta & x_1 \\ 
\sin\theta & \cos\theta & x_2 \\ 
0 & 0 & 1
\end{array}\right), 
\end{equation}
with $\theta\in(0,2\pi]$ and $x_1,x_2\in\mathbb{R}$, or 
\begin{equation}
g(a,\phi,\theta):=\left(\begin{array}{ccc}
\cos\theta & -\sin\theta & a\cos\phi \\ 
\sin\theta & \cos\theta & a\sin\phi \\ 
0 & 0 & 1
\end{array}\right).
\end{equation}
In semi-direct product notations, we may denote it as $(\mathbf{x},\mathbf{R}_\theta)$, where $\mathbf{x}=(x_1,x_2)^T$ and 
\[
\mathbf{R}_\theta=\left(\begin{array}{cc}
\cos\theta & -\sin\theta \\ 
\sin\theta & \cos\theta
\end{array}\right).
\]
The special Euclidean group $SE(2)$ is a unimodular group with the normalized Haar measure given by 
\[
\dd g=\frac{1}{4\pi^2}\dd x_1\dd x_2\dd\theta=\frac{1}{4\pi^2}a\dd a\dd\phi \dd\theta.
\]
The associated unified Parseval type formula on $SE(2)$ reads as 
\[
\int_{SE(2)}|f(g)|^2\dd g=\int_0^\infty\|\widehat{f}(p)\|_{\rm HS}^2p\dd p,
\]
and also the noncommutative Fourier reconstruction formula is given by
\[
f(g)=\int_0^\infty\mathrm{tr}\left[\widehat{f}(p)U_p(g)\right]p\dd p,
\]
for $f\in L^1\cap L^2(SE(2))$ and $g\in SE(2)$,   the linear operator $\widehat{f}(p)$ is given by 
\begin{equation}
\widehat{f}(p):=\int_{SE(2)}f(g)U_p(g^{-1})\D g=\int_{SE(2)}f(g)U_p(g)^*\D g,
\end{equation}
for every $p>0$.   

Suppose that $\mathbf{\Gamma}$ is a discrete and co-compact  subgroup of $SE(2)$ and $\mathcal{E}(\mathbf{\Gamma}):=(\psi_\ell)_{\ell\in\mathbb{I}}$ is an orthonormal basis of bounded functions for the Hilbert function space $L^2(\mathbf{\Gamma}\backslash SE(2),\mu)$.  
Let $f\in L^1\cap L^2(SE(d))$ such that $\widetilde{f}\in L^2(\mathbf{\Gamma}\backslash SE(2),\mu)$. Then  Theorem \ref{MainRecTh0.L2.K} implies that 
\begin{equation}
\widetilde{f}=\sum_{\ell\in\mathbb{I}}\left(\sum_{\gamma\in\mathbf{\Gamma}}\int_0^\infty\mathrm{tr}\left[\widehat{f}(p)Q_{\gamma\Omega}^\ell(p)\right]p\dd p\right)\psi_\ell.
\end{equation}
In addition, if $\widetilde{f}\in\mathcal{A}(\mathcal{E})$ then using Theorem \ref{ABC.Main},  we get 
\[
\widetilde{f}(\mathbf{\Gamma} g)=\sum_{\ell\in\mathbb{I}}\left(\sum_{\gamma\in\mathbf{\Gamma}}\int_0^\infty\mathrm{tr}\left[\widehat{f}(p)Q_{\gamma\Omega}^\ell(p)\right]p\dd p\right)\psi_\ell(\mathbf{\Gamma} g),
\]
for $\mu$-a.e. $\mathbf{\Gamma} g\in \mathbf{\Gamma}\backslash SE(2)$.

\subsection{The case $\mathbf{\Gamma}\backslash SE(3)$}

The 3D special Euclidean group, $SE(3)$, is the semidirect product of $\mathbb{R}^3$
with the 3D special orthogonal group $SO(3)$. That is,
\[
SE(3)=\mathbb{R}^3\rtimes SO(3)=SO(3)\ltimes\mathbb{R}^3.
\] 
We denote elements $g\in SE(3)$ as $g=(\mathbf{t},\mathbf{R})$ where $\mathbf{t}\in\mathbb{R}^3$ and $\mathbf{R}\in SO(3)$. For any $g=(\mathbf{t},\mathbf{R})$ and 
$g'=(\mathbf{t}',\mathbf{R}')\in SE(3)$ the group law is written as 
\[
g\circ g'=(\mathbf{t}+\mathbf{R}\mathbf{t}',\mathbf{RR}'),
\]
and 
\[
g^{-1}=(-\mathbf{R}^T\mathbf{t},\mathbf{R}^T)=(-\mathbf{R}^{-1}\mathbf{t},\mathbf{R}^{-1}).
\]
The special Euclidean group $SE(3)$ is a unimodular group with the normalized Haar measure given by 
\[
\D g=\dd\mathbf{t}\D\mathbf{R}=\frac{1}{8\pi^2}\D x\D y\D z\sin\beta \D\beta \D\alpha \D\gamma.
\]
The convolution of functions $f_k\in L^1(SE(3))$ with $k\in\{1,2\}$, is given by 
\[
(f_1\star f_2)(h)=\int_{SE(3)}f_1(g)f_2(g^{-1}\circ h)\D g,
\]
for $h\in SE(3)$.

The unified Fourier Plancherel/Parseval formula on the group $SE(3)$ is given by 
\[
\int_{SE(3)}|f(g)|^2\D g=\frac{1}{2\pi^2}\int_0^\infty\|\widehat{f}(p)\|_{\rm HS}^2p^2\D p,
\]
and also the noncommutative Fourier reconstruction formula is given by
\begin{equation}\label{FourierInv.SE3}
f(g)=\frac{1}{2\pi^2}\int_0^\infty\mathrm{tr}\left[\widehat{f}(p)U_p(g)\right]p^2\D p,
\end{equation}
for $f\in L^1\cap L^2(SE(3))$ and $g\in SE(3)$, where for $f\in L^1(SE(3))$ and $p>0$ 
the linear operator $\widehat{f}(p)$ is given by 
\begin{equation}
\widehat{f}(p):=\int_{SE(3)}f(g)U_p(g^{-1})\D g=\int_{SE(3)}f(g)U_p(g)^*\D g,
\end{equation}

Suppose that $\mathbf{\Gamma}$ is a discrete and co-compact  subgroup of $SE(3)$ and $\mathcal{E}(\mathbf{\Gamma}):=(\psi_\ell)_{\ell\in\mathbb{I}}$ is an orthonormal basis of bounded functions for the Hilbert function space $L^2(\mathbf{\Gamma}\backslash SE(3),\mu)$.  
Let $f\in L^1\cap L^2(SE(3))$ such that $\widetilde{f}\in L^2(\mathbf{\Gamma}\backslash SE(3),\mu)$. Then  Theorem \ref{MainRecTh0.L2.K} implies that 
\begin{equation}
\widetilde{f}=\sum_{\ell\in\mathbb{I}}\left(\sum_{\gamma\in\mathbf{\Gamma}}\int_0^\infty\mathrm{tr}\left[\widehat{f}(p)Q_{\gamma\Omega}^\ell(p)\right]p^2\dd p\right)\psi_\ell.
\end{equation}
In addition, if $\widetilde{f}\in\mathcal{A}(\mathcal{E})$ then using Theorem \ref{ABC.Main},  we get 
\[
\widetilde{f}(\mathbf{\Gamma} g)=\sum_{\ell\in\mathbb{I}}\left(\sum_{\gamma\in\mathbf{\Gamma}}\int_0^\infty\mathrm{tr}\left[\widehat{f}(p)Q_{\gamma\Omega}^\ell(p)\right]p^2\dd p\right)\psi_\ell(\mathbf{\Gamma} g),
\]
for $\mu$-a.e. $\mathbf{\Gamma} g\in \mathbf{\Gamma}\backslash SE(3)$.\\

{\bf Acknowledgments.}
This research is supported by University of Delaware Startup grants.
This work was supported in part by NUS Startup grants A-0009059-02-00 and A-
0009059-03-0. The findings and opinions expressed here are
only those of the authors, and not of the funding agencies.

\bibliographystyle{amsplain}

\end{document}